\documentclass[11pt]{amsart}

\usepackage{amssymb}
\usepackage{amsthm}
\usepackage{amsrefs}
\usepackage{enumitem}
\usepackage{todonotes}
\usepackage{comment}
\usepackage{xifthen}

\usetikzlibrary{calc}

\usepackage[colorlinks=true]{hyperref}

\theoremstyle{plain}
\newtheorem{theorem}{Theorem}[section]
\newtheorem{proposition}[theorem]{Proposition}
\newtheorem{lemma}[theorem]{Lemma}
\newtheorem{corollary}[theorem]{Corollary}
\newtheorem{conjecture}[theorem]{Conjecture}
\newtheorem{fact}[theorem]{Fact}

\theoremstyle{definition}
\newtheorem{definition}[theorem]{Definition}

\newtheorem{remark}[theorem]{Remark}
\newtheorem{claim}{Claim}[theorem]

\DeclareDocumentCommand \amal { m m O{} } {\ensuremath{\ifthenelse{\isempty{#3}}{\mathrm{Amal}}{\mathrm{Amal}_{#3}} (#1,#2)}}

\newcommand{\formula}[1]{\text{\fontfamily{cmss}\selectfont \textup{#1}}}
\newcommand{\mso}{\ensuremath{\operatorname{MSO}_{1}}}
\newcommand{\cmso}{\ensuremath{\operatorname{CMSO}_{1}}}
\newcommand{\var}{\ensuremath{\mathrm{Var}}}
\newcommand{\free}{\ensuremath{\mathrm{Free}}}
\newcommand{\bound}{\ensuremath{\mathrm{Bound}}}
\newcommand{\dash}{\nobreakdash-\hspace{0pt}}
\newcommand{\cl}{\operatorname{cl}}
\newcommand{\id}{\mathrm{Id}}
\newcommand{\cU}{\mathcal{U}}
\newcommand{\bN}{\mathbb{N}}

\tikzset{style1/.style={circle,fill=blue!20,minimum size=4pt}}
\tikzset{style2/.style={circle,fill=black!255,inner esp=0pt,minimum size=1pt}}
\tikzset{style3/.style={circle,fill=white!255,inner sep=0pt,minimum size=1pt}}

\makeatletter
\def\Ddots{\mathinner{\mkern1mu\raise\p@
\vbox{\kern7\p@\hbox{.}}\mkern2mu
\raise4\p@\hbox{.}\mkern2mu\raise7\p@\hbox{.}\mkern1mu}}
\makeatother

\makeatletter
\theoremstyle{plain}
\newtheorem*{rep@theorem}{\rep@title}
\newcommand{\newreptheorem}[2]{%
\newenvironment{rep#1}[1]{%
 \def\rep@title{#2 \ref{##1}}%
 \begin{rep@theorem}}%
 {\end{rep@theorem}}}
\makeatother

\newreptheorem{theorem}{Theorem}

\author[Funk]{Daryl Funk}

\author[Matthews]{Angus Matthews}

\author[Mayhew]{Dillon Mayhew}

\title[Myhill-Nerode for hypergraphs]
{Myhill-Nerode for hypergraphs and an application to gain-graphic matroids}

\begin{document}

\begin{abstract}
We present a Myhill-Nerode theorem for hypergraphs.
The theorem involves an operation which takes two input structures and produces a hypergraph as output.
Using this operation, we define a Myhill-Nerode-type equivalence relation and show that if a class of hypergraphs is definable in the counting monadic second-order logic of hypergraphs, then the equivalence relation has finite index.
We apply this tool to classes of gain-graphic matroids, and show that if the group $\Gamma$ is not uniformly locally finite, then the class of $\Gamma$\dash gain-graphic matroids is not monadically definable.
(A group is uniformly locally finite if, for every $k$, there is a maximum size amongst subgroups generated by at most $k$ elements.)
In addition, we define the conviviality graph of a group, and show that if the group $\Gamma$ has an infinite conviviality graph, then the class of $\Gamma$\dash gain-graphic matroids is not monadically definable.
This will be useful in future constructions.
\end{abstract}

\maketitle

\section{Introduction}

The Myhill-Nerode theorem is a key tool for demonstrating that a language of strings is not regular.
Since a language is regular if and only if it can be defined by a sentence in the monadic second-order language of strings, this allows us to prove that certain languages are not monadically definable.

In this article, we present a version of the Myhill-Nerode theorem for hypergraphs, which we will use to prove that certain classes of hypergraphs are not monadically definable.
In particular, we will prove non-definability results when the hypergraphs are gain-graphic matroids.
In the classical Myhill-Nerode theorem, $\mathcal{L}$ is a language of strings and two strings,  $\mathbf{w}_{1}$ and $\mathbf{w}_{1}$, are equivalent relative to $\mathcal{L}$, if, for every possible string $\mathbf{z}$, either both of the concatenations $\mathbf{w}_{1}\mathbf{z}$ and $\mathbf{w}_{2}\mathbf{z}$ are in $\mathcal{L}$, or neither is.
The language is regular if and only if this equivalence relation has finite index.
To present a theorem of this type for hypergraphs, we need an operation for hypergraphs that plays the same role as concatenation in the classical theorem.
Let $C$, $U$, and $V$ be finite sets where $U\cap V = \emptyset$.
Let $c$ be a function from $2^{U}$ to $C$ and let $d$ be a function from $2^{V}\times C$ to $\{0,1\}$.
We can glue $(U,c)$ and $(V,d)$ together to produce a hypergraph $(U,c)\boxplus (V,d)$ with the ground set $U\cup V$, where the hyperedges are the sets of the form $X\cup Y$ satisfying $X\subseteq U$, $Y\subseteq V$, and $d(Y,c(X))=1$.
This operation generalises natural matroid operations such as direct sums, $2$\dash sums, and proper amalgams.

Let $\mathcal{M}$ be a class of hypergraphs and let $C$ be a finite set.
We define an equivalence relation on all pairs $(U,c)$, where $U$ is a finite set and $c$ is a function from $2^{U}$ to $C$.
We write $(U_{1},c_{1})\sim_{\mathcal{M}, C} (U_{2},c_{2})$ if, for every possible pair $(V,d)$, we have that $(U_{1},c_{1})\boxplus (V,d)$ is in $\mathcal{M}$ if and only if $(U_{2},c_{2})\boxplus (V,d)$ is in $\mathcal{M}$.
Lemma~\ref{thm:sharp-Myhill-Nerode} tells us that if $\mathcal{M}$ is definable in the monadic second-order logic of hypergraphs, then $\sim_{\mathcal{M}, C}$ has finite index.

We note some results in the literature that have commonalities with our version of Myhill-Nerode.
First, Lemma~\ref{thm:sharp-Myhill-Nerode} is a generalisation of Lemmas~1.3 and 1.4 in~\cite{MNW18}, since those lemmas are specific to particular types of matroid summing operations.
Lemma~\ref{thm:sharp-Myhill-Nerode} is independent from the work in~\cite{vBDFGR15}*{Corollary~3}, since the notion of a hypergraph sum in that work bounds the number of hyperedges that intersect both sides of the sum.
Our notion of a sum does not require any such bound.
Our lemma is also independent of the tool created by Kotek and Makowsky~\cite{KM14}*{Theorem~3.5}, as the binary operation $\boxplus$ is not \emph{smooth} (using their language).

Starting in Section~\ref{sec:amalgams}, we apply Lemma~\ref{thm:sharp-Myhill-Nerode} to questions of monadic definability for classes of matroids.
We use $\cmso$ to refer to the \emph{counting monadic second-order logic} of matroids (and more generally, hypergraphs).
This language has predicates that let us say when a subset of the domain has cardinality congruent to $p$ modulo $q$, for any appropriate pair $p$ and $q$.
The fragment of $\cmso$ that does not use these predicates is denoted by $\mso$.
(This language has been at various times denoted by $\mathrm{MSOL}$, $\mathrm{MS}_{M}$, and $\mathrm{MS}_{0}$.)

Courcelle's Theorem for graphs~\cite{Cou90} provided much of the original motivation for studying the monadic second-order logic of graphs.
There are matroidal analogues of this theorem which provide a similar motivation for wanting to know when a property of matroids can be defined in $\cmso$~\cites{FMN22, Hli03a, Kra12, Str11}.
In particular, we are motivated to understand which minor-closed classes of matroids are definable in $\cmso$.
(This question is not interesting in the context of graphs, because the Robertson-Seymour Theorem shows that any minor-closed class of graphs has finitely many excluded minors and is therefore monadically definable.)
For every group $\Gamma$ there is a corresponding class of \emph{$\Gamma$-gain-graphic matroids}, just as for every field $\mathbb{F}$, there is a class of $\mathbb{F}$\dash representable matroids.
In a representable matroid the elements of the matroid are associated with vectors over $\mathbb{F}$, and in the gain-graphic case they are associated with elements from $\Gamma$.
In either case, we are essentially providing algebraic coordinates that specify the relations between matroid elements.
The two types of classes play central and symmetric roles in structural matroid theory~\cites{GGW13, GNW24, KK82}.
Despite this, gain-graphic classes have not received as much attention as representable classes.
A result by Mayhew, Newman, and Whittle shows that the class of $\mathbb{F}$\dash representable matroids is $\cmso$\dash definable if and only if $\mathbb{F}$ is finite.
The following conjectures, from~\cite{FMN21}, were made by analogy with this result.

\begin{conjecture}\label{conj:finite-definable}
Let $\Gamma$ be a finite group.
The class of $\Gamma$\dash gain-graphic matroids is \mso\dash definable.
\end{conjecture}

\begin{conjecture}\label{conj:infinite-nondefinable}
Let $\Gamma$ be an infinite group.
The class of $\Gamma$\dash gain-graphic matroids is not \mso\dash definable.
\end{conjecture}

The first and third authors have shown that Conjecture~\ref{conj:finite-definable} is true.
However, this positive result requires a structural theorem giving us control over the representations of frame matroids by biased graphs.
This will allow us to construct a monadic transduction taking frame matroids as input and producing their biased-graphic representations as output.
The proof of this structural theorem will be lengthy and is still work in progress, so for now we claim this definability result without proof.

In contrast to this positive result, Conjecture~\ref{conj:infinite-nondefinable} is false.
If, for example, $\Gamma$ is the direct product of an infinite number of copies of $\mathbb{Z}_{2}$, then the class of $\Gamma$\dash gain-graphic matroids is \mso\dash definable.
In fact, the first and third authors and Ben-Shahar have built a hierarchy of infinite groups, each of which gives rise to an \mso\dash definable class of gain-graphic matroids.
This proof again requires the structural representation theorem, so we do not provide a proof here.

In this article we instead present some negative results: proving that for certain infinite groups $\Gamma$, the class of $\Gamma$\dash gain-graphic matroids is not \cmso\dash definable.
We do so by exhibiting representatives to demonstrate that there are infinitely many equivalence classes under the relation $\sim_{\mathcal{M}, C}$.

Let $\Gamma$ be a (multiplicative) group.
Then $\Gamma$ has \emph{finite exponent} if there is some positive integer $p$ such that $h^{p}$ is the identity for every $h\in \Gamma$.
If every finite subset of $\Gamma$ generates a finite subgroup then $\Gamma$ is \emph{locally finite}.
Assume there is a function $g_{\Gamma}$ taking positive integers to positive integers such that any subgroup of $\Gamma$ generated by at most $k$ elements has size at most $g_{\Gamma}(k)$.
In this case $\Gamma$ is \emph{uniformly locally finite}.
It is immediate that a uniformly locally finite group is locally finite.
Moreover, any subgroup generated by a single element has order at most $g_{\Gamma}(1)$, and hence we see that $\Gamma$ has finite exponent.
So uniform local-finiteness is a sufficient condition for $\Gamma$ to be locally finite with finite exponent.
The restricted version of the famous Burnside problem says that it is also necessary.
This result is known to be true thanks to the work of Zel\cprime manov~\cites{Zel90, Zel91}.
To reiterate: a group is uniformly locally finite if and only if it has finite exponent and is locally finite.

\begin{reptheorem}{thm:uniformly-locally-finite}
Let $\Gamma$ be an infinite group that is not uniformly locally finite.
The class of $\Gamma$\dash gain-graphic matroids is not \cmso\dash definable.
\end{reptheorem}

It is reasonable to ask if the other direction of Theorem~\ref{thm:uniformly-locally-finite} holds: is it true for any $\Gamma$ that if the class of $\Gamma$\dash gain-graphic matroids is not \cmso\dash definable, then $\Gamma$ is not uniformly locally finite?
This is more naturally asked in the contrapositive: if $\Gamma$ is uniformly locally finite, then is the class of $\Gamma$\dash gain-graphic matroids \cmso\dash definable?
A counting argument shows that this cannot be the case.
There are uncountably many infinite groups that are uniformly locally finite.
It is not hard to show that in addition, there are uncountably many distinct classes of $\Gamma$\dash gain-graphic matroids, where $\Gamma$ is uniformly locally finite.
Since there are only countably many \cmso\dash sentences, it follows that the converse of Theorem~\ref{thm:uniformly-locally-finite} cannot hold.

In fact, it is possible to go further, and explicitly construct an infinite group $\Gamma$ such that $\Gamma$ is uniformly locally finite and the class of $\Gamma$\dash gain-graphic matroids is not \cmso\dash definable.
We do not provide the description of such a group here, but we develop a tool that will aid in its construction.
Let $F$ be a finite subgroup of an infinite group $\Gamma$.
The \emph{$F$\dash conviviality graph} of $\Gamma$ carries information about how copies of $F$ are embedded in $\Gamma$: specifically, how these copies of $F$ relate to other finite subgroups of $\Gamma$.
As a second application of Lemma~\ref{thm:sharp-Myhill-Nerode} we prove the following result.

\begin{reptheorem}{thm:conviviality-graph}
Let $\Gamma$ be a group.
If $\Gamma$ has a finite subgroup $F$ such that the $F$\dash conviviality graph of $\Gamma$ is infinite, then the class of $\Gamma$\dash gain-graphic matroids is not \cmso\dash definable.
\end{reptheorem}

The structure of the paper is as follows: In Section~\ref{sec:prelims} we cover some fundamental notions of hypergraphs, monadic logic, matroids, and biased graphs.
We use ultrafilters and ultraproducts to simplify the proof of Theorem~\ref{thm:uniformly-locally-finite}, and these concepts are explained in Section~\ref{subsec:ultrapowers}.
Section~\ref{sec:myhill-nerode} gives the context and proof for an analogue of the Myhill-Nerode characterisation for monadically defined classes of hypergraphs.
Section~\ref{sec:amalgams} is a purely matroidal section, establishing properties of the `gluing' operation that we use.
In Section~\ref{sec:ULF} we use our assembled tools to prove Theorem~\ref{thm:uniformly-locally-finite}.
Section~\ref{sec:convivial} introduces the notion of an $F$\dash conviviality graph and proves Theorem~\ref{thm:conviviality-graph}.

\section{Preliminaries}
\label{sec:prelims}

We write $\bN$ for the set of positive integers.
If $n$ is in $\bN$, we write $[n]$ for $\{1,2,\ldots, n\}$.
We write $2^{U}$ for the power set of the set $U$.
If $I$ is a set and $i$ is an element then we write $I+i$ for $I\cup\{i\}$.
We regard each function as a set of ordered pairs.
So if $\sigma\colon X\to Y$ is a function and $x$ is an element not in $X$ then $\sigma+(x,y)$ is the function with domain $X+x$ which takes each element of $x$ to its image under $\sigma$ and which takes $x$ to $y$.
Graphs may have loops and parallel edges.
If $G$ is a graph and $X$ is a set of edges, then $G[X]$ is the subgraph with $X$ as its set of edges.
The vertices of $G[X]$ are exactly the vertices of $G$ that are incident with at least one edge in $X$.
We will very frequently blur the distinction between sets of edges and subgraphs.
For example, a cycle may be a set of edges or it may be a subgraph, according to which is more convenient for us.

\subsection{Hypergraphs}
A \emph{hypergraph} consists of a finite set $E$ and a collection $\mathcal{I}$ of subsets of $E$.
We refer to $E$ as the \emph{ground set} of the hypergraph, and call the members of $\mathcal{I}$ the \emph{hyperedges}.

The foundations of matroid theory can be found in Oxley~\cite{Oxl11}.
A \emph{matroid} is a hypergraph where the collection of hyperedges is non-empty and is closed downwards under subset containment, and furthermore, whenever $I$ and $J$ are hyperedges satisfying $|I|<|J|$, then there is an element $e\in J-I$ such that $I+e$ is a hyperedge.
The hyperedges of a matroid are called \emph{independent sets}.
The \emph{dependent} subsets are the subsets of the ground set that are not independent.
A dependent subset that does not properly contain a dependent subset is a \emph{circuit}.
A (matroidal) \emph{loop} is an element $e$ such that $\{e\}$ is a circuit.
A \emph{coloop} is an element that is in no circuit.
A matroid is \emph{simple} if every subset of size at most two is independent.

Let $M=(E,\mathcal{I})$ be a matroid.
If $X$ is a subset of $E$ then the \emph{rank} of $X$, written $r(X)$, is the maximum cardinality of an independent subset of $X$.
Thus $r(X)=|X|$ if and only if $X$ is independent.
We write $M|X$ for the matroid
\[
(X,\{I\in \mathcal{I}\colon I\subseteq X\}).
\]
and we refer to this as the \emph{restriction} of $M$ to $X$.
A \emph{flat} is a subset $F\subseteq E$ such that $r(F+x)>r(F)$ for every $x\in E-F$.
An intersection of flats is also a flat.
The \emph{closure} of $X$, written $\cl(X)$, is the intersection of all flats that contain $X$.
Assume $X$ and $Y$ are disjoint sets.
Then $r(X)+r(Y)\geq r(X\cup Y)$ by submodularity of the rank function~\cite{Oxl11}*{Lemma~1.3.1}.
If $r(X)+r(Y)=r(X\cup Y)$ then the pair of sets is \emph{skew}.
This is the case if and only if there is no circuit contained in $X\cup Y$ that contains elements of both $X$ and $Y$~\cite{Oxl11}*{Proposition~4.2.1}.
A \emph{separation} of the matroid is a partition of $E$ into a skew pair of non-empty sets.

Let $E$ be a finite set of vectors from a vector space $V$.
Declare a subset of $E$ to be a hyperedge if and only if it is linearly independent.
The resulting hypergraph is a \emph{representable matroid}.

\subsection{Monadic second-order logic}
In this section we construct \emph{counting monadic second-order logic} for hypergraphs, which we denote by $\cmso^{\textup{hyp}}$.
In the context of this article, monadic second-order logic always applies to hypergraphs, so we omit the superscript and write $\cmso$.
Formulas will be constructed using variables from the set $\{Z_{1},Z_{2},\ldots\}$.
The \emph{atomic formulas} are as follows.
\begin{enumerate}[label = $\bullet$, labelindent = 0em, labelwidth = 1em, labelsep* = 0.5em, leftmargin =!]
\item $Z_{i}\subseteq Z_{j}$ is an atomic formula for $i,j \in\bN$.
We declare
\[\var(Z_{i}\subseteq Z_{j})=\free(Z_{i}\subseteq Z_{j})=\{Z_{i},Z_{j}\}.\]
\item $\formula{hyp}(Z_{i})$ is an atomic formula for $i\in \bN$.
In this case
\[\var(\formula{hyp}(Z_{i}))=\free(\formula{hyp}(Z_{i}))=\{Z_{i}\}.\]
\item $|Z_{i}|\equiv p\bmod{q}$ is an atomic formula for $i\in\bN$, where $p$ and $q$ are integers satisfying $q>1$ and $0\leq p < q$.
We define
\[\var(|Z_{i}|\equiv p\bmod{q})=\free(|Z_{i}|\equiv p\bmod{q})=\{Z_{s}\}.\]
\end{enumerate}
Any \emph{formula} in \cmso\ is built using the following rules.
\begin{enumerate}[label = $\bullet$]
\item Every atomic formula is a formula.
\item If $\psi$ is a formula then $\neg\psi$ is a formula and $\var(\neg\psi)=\var(\psi)$ while $\free(\neg\psi)=\free(\psi)$.
\item If $\psi$ is a formula and $Z_{s}$ is in $\free(\psi)$ then $\exists Z_{s}\psi$ is a formula and $\var(\exists Z_{s}\psi)=\var(\psi)$ while $\free(\exists Z_{s}\psi)=\free(\psi)-\{Z_{s}\}$.
\item Assume that $\psi_{1}$ and $\psi_{2}$ are formulas such that
\[
(\var(\psi_{i})-\free(\psi_{i}))\cap \free(\psi_{3-i})=\emptyset
\]
for $i=1,2$.
Then $\psi_{1}\land\psi_{2}$ is a formula.
We declare
\[
\var(\psi_{1}\land\psi_{2})=\var(\psi_{1})\cup \var(\psi_{2})\ \quad\text{and}\]
\[\free(\psi_{1}\land\psi_{2})=\free(\psi_{1})\cup \free(\psi_{2}).\]
\end{enumerate}
If $\varphi$ is a \cmso\dash formula, then any variable in $\free(\varphi)$ is a \emph{free variable} of $\varphi$.
Any variable in $\var(\varphi)-\free(\varphi)$ is a \emph{bound variable} of $\varphi$.
We use $\bound(\varphi)$ to denote the set of bound variables in $\varphi$.
If $\var(\varphi)=\bound(\varphi)$ then $\varphi$ is a \emph{\cmso\dash sentence}.

The collection of \cmso\dash formulas that we construct without using any atomic formula of the form $|\cdot|\equiv p\bmod{q}$ is \emph{monadic second-order logic} for hypergraphs, which we denote by \mso.
Let $\delta$ be a positive integer.
We say that a \cmso\dash formula is \emph{$\delta$\dash confined} if it can be constructed without using a predicate of the form $|\cdot|\equiv p\bmod{q}$ where $q>\delta$.
Note that a formula is $1$\dash confined if and only if it is an \mso\dash formula.

We have now discussed the syntax of monadic second-order logic for hypergraphs.
Let us move to the semantics.
Let $\varphi$ be a \cmso\dash formula and let $M=(E,\mathcal{I})$ be a hypergraph.
An \emph{interpretation} of $\varphi$ in $M$ is a function $\theta\colon \free(\varphi)\to 2^{E}$.
We define what it means for $\varphi$ to be \emph{satisfied} by $(M,\theta)$.
If $\varphi$ is $Z_{i}\subseteq Z_{j}$ then $\varphi$ is satisfied if $\theta(Z_{i})\subseteq \theta(Z_{j})$.
If $\varphi$ is $\formula{hyp}(Z_{i})$ then $\varphi$ is satisfied if $\theta(Z_{i})$ is in $\mathcal{I}$.
Next, if $\varphi$ is $|Z_{i}|\equiv p\bmod{q}$ then $\varphi$ is satisfied if $|\theta(Z_{j})|$ is congruent to $p$ modulo $q$.

Now we assume that $\varphi$ is not atomic.
We define satisfaction recursively.
If $\varphi$ is $\neg\psi$ then $\varphi$ is satisfied if $\psi$ is \emph{not} satisfied by $(M,\theta)$.
If $\varphi$ is $\exists Z_{s}\psi$, then $\varphi$ is satisfied if there exists a subset $X\subseteq E$ such that $\psi$ is satisfied by $(M,\theta+(Z_{s},X))$.
Finally, if $\varphi$ is $\psi_{1}\land\psi_{2}$, then we let $\theta_{i}$ be the restriction of $\theta$ to $\free(\psi_{i})$ for $i=1,2$.
Now $\varphi$ is satisfied if $\psi_{i}$ is satisfied by $(M,\theta_{i})$ for all $i\in\{1,2\}$.
If $\varphi$ is a \cmso\dash sentence, then $\varphi$ has no free variables.
In this case, an interpretation is the empty function, and we may speak of $\varphi$ being satisfied by $M$, rather than $(M,\theta)$.

Let $\mathcal{M}$ be a class of hypergraphs.
Assume there is a \cmso\dash sentence $\varphi$ such that $\varphi$ is satisfied by a hypergraph if and only if that hypergraph belongs to $\mathcal{M}$.
In this case we say that $\mathcal{M}$ is \emph{\cmso\dash definable}.
If $\varphi$ is a \mso\dash sentence then $\mathcal{M}$ is \emph{\mso\dash definable}.

\subsection{Gain-graphic matroids}

Let $G$ be a graph with edge-set $E$ and vertex-set $V$.
A \emph{bicycle} of $G$ is a subset $X\subseteq E$ such that $X$ is minimal with respect to $G[X]$ being connected and containing at least two cycles.
Every bicycle is a \emph{handcuff} or a \emph{theta subgraph}.
The first of these consists of two cycles with at most one vertex in common, along with a unique minimal path joining the two cycles.
(Note that this path may consist of a single vertex that is in both cycles --- in this case the handcuff is \emph{tight} and otherwise it is a \emph{loose}.)
A theta subgraph consists of two distinct vertices and three pairwise internally-disjoint paths that join them.
A \emph{linear class} is a set $\mathcal{B}$ of cycles such that no theta subgraph in $G$ contains exactly two cycles in $\mathcal{B}$.
In this case, we say that $(G,\mathcal{B})$ is a \emph{biased graph}.
Any cycle that belongs to $\mathcal{B}$ is said to be \emph{balanced}.

When $\Omega=(G,\mathcal{B})$ is a biased graph we can define $F(\Omega)$, the \emph{frame matroid} of $\Omega$.
The ground set of $F(\Omega)$ is the edge-set of $G$.
The circuits are the edge-sets of balanced cycles and the edge-sets of bicycles that contain no balanced cycle.
Note that a loop edge in the graph $G$ will only be a loop in the matroid $F(\Omega)$ if that edge comprises a balanced cycle.

The rank of $F(\Omega)$ is obtained by subtracting the number of connected components in $\Omega$ that contain no unbalanced cycles from the number of vertices in $\Omega$.
A \emph{line} of a matroid is a rank\dash $2$ flat.
The intersection of two distinct lines has rank at most one.
A line is \emph{long} if it contains at least four rank\dash $1$ flats.

\begin{proposition}
\label{prop:long-lines}
Let $\Omega=(G,\mathcal{B})$ be a biased graph.
Assume that the element $e$ is contained in two distinct long lines of $F(\Omega)$.
Then $e$ is a loop edge of $G$.
\end{proposition}

\begin{proof}
We can let $\{a,b,c,e\}$ and $\{x,y,z,e\}$ be sets such that any $3$\dash element subset of either is a circuit.
We can also assume that no rank\dash $2$ flat contains both sets.
Assume that $\{a,b,e\}$ is the edge-set of a balanced cycle in $\Omega$.
Then $a$ and $b$ form a path of two edges, and $e$ must form a circuit with these edges.
This is only possible if $\{a,b,c\}$ is also the set of edges in a balanced cycle, which means $c$ and $e$ are parallel edges.
Therefore $\{a,b,c,e\}$ is the edge-set of a theta subgraph, but the cycle comprising $c$ and $e$ is not balanced, since $\{a,c,e\}$ is a circuit.
Therefore this theta subgraph contains exactly two balanced cycles, and we have a contradiction.
Exactly the same argument shows that no three edges from $\{a,b,c,e\}$ or $\{x,y,z,e\}$ form a balanced cycle.
This means that $\{a,b,c\}$ is a bicycle that contains no balanced cycle.
Subsequently $G[\{a,b,c\}]$ contains exactly two vertices.
Let these vertices be $p$ and $q$.
Then $e$ is not incident with any vertex not in $\{p,q\}$.
The same argument shows that $G[\{x,y,z\}]$ has exactly two vertices, $s$ and $t$, and $e$ is not incident with any vertex not in $\{s,t\}$.
If $\{s,t\}=\{p,q\}$ then there is a rank\dash $2$ flat that contains $\{a,b,c,x,y,z,e\}$, contrary to assumption.
Therefore $\{s,t\}$ and $\{p,q\}$ have at most one vertex in common, which means that $e$ is incident with at most one vertex.
Thus $e$ is a loop, as we claimed.
\end{proof}

Let $G$ be a graph and let $\Gamma$ be a (multiplicative) group.
A \emph{$\Gamma$\dash gaining} of $G$ is a function $\sigma$ that takes as input any triple $(e,u,v)$ such that either $e$ is an edge joining the distinct vertices $u$ and $v$, or $e$ is a loop incident with $u$ and $u=v$.
The codomain of $\sigma$ is $\Gamma$.
We require that if $u\ne v$ then $\sigma(e,u,v)=\sigma(e,v,u)^{-1}$.
Now we say that $(G,\sigma)$ is a \emph{$\Gamma$\dash gain-graph}.
If $\Omega=(G,\mathcal{B})$ is a biased graph, and there exists a $\Gamma$\dash gaining $\sigma$ such that $\mathcal{B} = \mathcal{B}(\sigma)$, then $\Omega$ is \emph{$\Gamma$\dash gainable}.
If $e$ is a loop edge of $G$ incident with the vertex $u$ and $\sigma(e,u,u)$ is not the identity, then we say $e$ is an \emph{unbalanced loop}.

Let $W$ be a walk of $G$ and let $v_{0},e_{0},v_{1},e_{1},\ldots, e_{n-1},v_{n}$ be the sequence of vertices and edges in $W$.
Then $\sigma(W)$ is
\[
\sigma(e_{0},v_{0},v_{1})\sigma(e_{1},v_{1},v_{2})\cdots \sigma(e_{n-1},v_{n-1},v_{n}).
\]
Let $\mathcal{B}(\sigma)$ be the collection of cycles $C$ in $G$ such that $\sigma(C)$ is the identity of $\Gamma$.
(If this condition holds, then it will hold no matter which starting point and orientation of the cycle is chosen.)
Then $\mathcal{B}(\sigma)$ is a linear class of cycles~\cite{Zas89}*{Proposition~5.1}.
We write $F(G,\sigma)$ to denote the frame matroid $F(G,\mathcal{B}(\sigma))$.
Any such matroid is said to be \emph{$\Gamma$\dash gain-graphic} matroid.

If $\sigma$ is a $\Gamma$\dash gaining and $\rho$ is a function from $V(G)$ to $\Gamma$, then $\sigma^{\rho}$ is the $\Gamma$\dash gaining that takes $(e,u,v)$ to
\[\rho(u)^{-1}\sigma(e,u,v)\rho(v)\]
when $u\ne v$.
(And which takes any tuple $(e,u,u)$ to $\sigma(e,u,u)$.)
It is easy to see that $\mathcal{B}(\sigma^{\rho}) = \mathcal{B}(\sigma)$, and therefore $F(G,\sigma^{\rho}) = F(G,\sigma)$.
We say that $\sigma^{\rho}$ is obtained from $\sigma$ by \emph{switching}.
Let $T$ be a maximal forest of $G$.
By performing an appropriately chosen switching, we can obtain a $\Gamma$\dash gaining that takes $(e,u,v)$ to the identity of $\Gamma$ whenever $e$ is an edge in $T$~\cite{Zas89}*{Lemma~5.3}.

\subsection{Ultrapowers}
\label{subsec:ultrapowers}
Here we will give a brief description of the theory of ultraproducts. All this material is standard; see for instance \cite{H97}*{Section 8.5}.

\begin{definition}
    An \emph{ultrafilter} on $\bN$ is a set $\cU \subseteq \mathcal{P}(\bN)$ such that
    \begin{enumerate}[label = \textup{(\roman*)}]
        \item For all $S \in \cU$, for all $S \subseteq T \subseteq \bN$, $T \in \cU$.
        \item For all $S_1,S_2 \in \cU$, $S_1 \cap S_2 \in \cU$.
        \item For all $S \subseteq \bN$, exactly one of $S$ and $\bN \setminus S$ is in $\cU$. In particular, $\emptyset \notin \cU$.
    \end{enumerate}
\end{definition}

\begin{remark}
    For any $n \in \bN$, $\{S : n \in S \subseteq \bN\}$ is an ultrafilter on $\bN$. We say an ultrafilter $\cU$ on $\bN$ is \emph{non-principal} if it is not of the above form. By the axiom of choice, such an ultrafilter must exist.
\end{remark}
\begin{remark}\label{rmk:non-principal-is-cofinite}
    An ultrafilter $\cU$ on $\bN$ is non-principal iff it contains no finite set, or equivalently iff it contains every cofinite set. This follows from properties (ii) and (iii) of the definition.
\end{remark}

\begin{definition}
    Let $(\Gamma_i)_{i \in \bN}$ be a sequence of groups, and let $\cU$ be a non-principal ultrafilter on $\bN$. Then there is a relation $\sim_{\cU}$ on $\Pi_i \Gamma_i$ defined by 
    \[(a_i) \sim_{\cU} (b_i) \Leftrightarrow \{i : a_i = b_i\} \in \cU\]
    It follows from the definition of an ultrafilter that $\sim_{\cU}$ is a $\Pi_i \Gamma_i$-invariant equivalence relation. Now the ultraproduct is defined as $\Pi_i \Gamma_i / \cU = \Pi_i \Gamma_i/\hspace{-0.5em}\sim_{\cU}$.
    It is easy to see that $\Pi_i \Gamma_i / \cU$ is still a group. We will only be interested in the case where all the $G_i$ are equal. In this case, the ultraproduct is referred to as an \emph{ultrapower}, and the ultrapower of $\Gamma$ is written as $\Gamma^{\cU}$.
\end{definition}

\begin{definition}
Let 
\[\{s_i(x_1, x_1', \ldots, x_l,x_l') : i \in S\}\quad\text{and}\quad \{t_i(x_1, x_1', \ldots, x_l,x_l') : i \in T\}\]
be two sets of strings in $\{x_1, x_1', \ldots, x_l,x_l'\}^*$. Given a group $\Gamma$, and elements $g_1, \ldots, g_l \in \Gamma$, there is a natural evaluation map taking an $s_i$ or $t_i$ and returning an element of $\Gamma$. We will write this $s_i(g_1,g_1^{-1}, \ldots, g_l,g_l^{-1})$. Given some $l$, a group $\Gamma$ and two such sets of strings $\{s_i\}$, $\{t_i\}$, let us say $\Gamma$ \emph{solves the pair} $(\{s_i\}, \{t_i\})$ if there are elements $g_1, \ldots, g_l$ such that for all $i \in S$,
\[s_i(g_1,g_1^{-1}, \ldots, g_l,g_l^{-1}) = \id\]
and for all $i \in T$,
\[t_i(g_1,g_1^{-1}, \ldots, g_l,g_l^{-1}) \neq \id\]
\end{definition}
The next proposition is a special case of \L{}o\'s's Theorem.
We include a proof to make this section more self-contained.
\begin{proposition}\label{prop:Los-Thm}
    Fix some $l$, and let $\{s_i : i \in S\}$, $\{t_i : i \in T\}$ be two finite lists of strings from $\{x_1, x_1', \ldots, x_l,x_l'\}^*$. 
Let $\Gamma$ be a group and $\cU$ be a non-principal ultrafilter on $\bN$. Then the following are equivalent:
    \begin{enumerate}[label = \textup{(\roman*)}]
        \item $\Gamma$ solves the pair $(\{s_i : i \in S\},\{t_i : i \in T\})$.
        \item $\Gamma^\cU$ solves the pair $(\{s_i : i \in S\},\{t_i : i \in T\})$.
    \end{enumerate}
\end{proposition}
\begin{proof}
    We first do the forward direction. Let $\alpha$ be the map $\Gamma \rightarrow \Gamma^{\cU}$ given by
    \[
    g \mapsto (g,g,g, \ldots)/\!\sim_{\cU}
    \]
    Note that $\alpha$ is injective, as by definition $\emptyset \notin \cU$, so for any $g_1 \neq g_2$, we have $(g_1,g_1,g_1 \ldots) \not\sim_{\cU} (g_2,g_2,g_2 \ldots)$.
    Thus, $\alpha$ is a group embedding. So given $g_1, \ldots, g_l \in \Gamma$ satisfying condition (i), $\alpha(g_1), \ldots, \alpha(g_l)$ satisfies condition (ii).

    Now we do the backward direction. Let $(g_k^{(1)})_{k \in \bN}, \ldots, (g_k^{(l)})_{k \in \bN} \in \Gamma^{\cU}$ satisfy condition (ii). Thus, for each $1 \leq i \leq n$,
    \[
    S_i = \{k : s_i(g_k^{(1)}, (g_k^{(1)})^{-1}, \ldots, g_k^{(l)}, (g_k^{(l)})^{-1}) = \id\} \in \cU
    \]
    and for each $1 \leq j \leq m$,
    \[
    T_i = \{k : t_j(g_k^{(1)}, (g_k^{(1)})^{-1}, \ldots, g_k^{(l)}, (g_k^{(l)})^{-1}) \neq \id\} \in \cU
    \]
    By property (ii) in the definition of ultrafilters, $\bigcap S_i \cap \bigcap T_j \in \cU$, and is therefore nonempty. Let $k$ be an element in this set.
It then follows that $g_k^{(1)}\ldots g_k^{(l)}$ satisfy condition (i).
\end{proof}

\begin{corollary}
    For any group $\Gamma$ and non-principal ultrafilter $\cU$ on $\bN$, a biased graph $\Omega$ is $\Gamma$\dash gainable if and only if it is $\Gamma^{\cU}$-gainable.
\end{corollary}
\begin{proof}
    It is straightforward to see that $\Omega$ is $\Gamma$\dash gainable if and only if there exist elements of $\Gamma$ satisfying a certain list of equations and inequations. By the above proposition, such elements exist in $\Gamma$ if and only if they exist in $\Gamma^{\cU}$.
\end{proof}
\begin{corollary}\label{cor:gain-matroids-are-same}
    For any group $\Gamma$ and non-principal ultrafilter $\cU$ on $\bN$, the class of $\Gamma$\dash gain-graphic matroids is equal to the class of $\Gamma^{\cU}$\dash gain-graphic matroids.
\end{corollary}

\begin{corollary}\label{cor:ulf-iff-up-lf}
    For any group $\Gamma$ and non-principal ultrafilter $\cU$ on $\bN$, the following statements are equivalent.
    \begin{enumerate}[label = \textup{(\roman*)}]
    \item $\Gamma$ is uniformly locally finite.
    \item $\Gamma^{\cU}$ is uniformly locally finite
    \item $\Gamma^{\cU}$ is locally finite.
    \end{enumerate}
\end{corollary}

\begin{proof}
    First, $\text{(i)} \Rightarrow \text{(ii)}$. Assume for a contradiction that $\Gamma^{\mathcal{U}}$ is not uniformly locally finite, but $\Gamma$ is.
There exists a $K$ such that $\Gamma^{\mathcal{U}}$ contains arbitrarily large subgroups generated by at most $K$ elements.
Let $g_{\Gamma}(K)$ be the maximum size of a subgroup of $\Gamma$ generated by at most $K$ elements.
Let $h_{1},\ldots, h_{K}$ be elements of $\Gamma^{\mathcal{U}}$ that generate a subgroup of more than $g_{\Gamma}(K)$ elements.
Then there is a finite list of inequalities certifying that $|\langle \{h_{1},\ldots, h_{K}\}\rangle | > g_{\Gamma}(K)$. We apply Proposition~\ref{prop:Los-Thm} to this list of inequalities to deduce that there are elements $g_1, \ldots, g_K \in \Gamma$ also satisfying those inequalities. But then $|\langle \{g_{1},\ldots, g_{K}\}\rangle | > g_{\Gamma}(K)$, yielding a contradiction.

    The fact that (ii) implies (iii) is obvious from definitions, as discussed in the introduction.
    To prove $\text{(iii)}\Rightarrow\text{(i)}$, we assume $\Gamma$ is not uniformly locally finite, and prove that $\Gamma^{\cU}$ is not locally finite.
    Our assumption means that there is some positive integer $K$ such that there is a sequence of tuples $(g_{i,j})_{1\leq i \leq K, j \in \bN}$ in $\Gamma$ generating subgroups of strictly increasing size. Consider
    \[\{(g_{i,1}, g_{i,2}, \ldots)/\!\sim_{\cU} : 1 \leq i \leq K\} \subseteq \Gamma^{\cU}\]
    Call these elements $h_1, \ldots, h_K$. Suppose they generate a finite subgroup of $\Gamma^{\cU}$, and let the size of this subgroup be $R$. This finiteness is witnessed by a finite sequence of equalities $\{s_j(h_1,h_1^{-1}, \ldots, h_K,h_K^{-1}) = \id : j \in J\}$ (e.g.\ let these equalities imply all the equalities in the Cayley table of this group, thus forcing that any set satisfying those equalities is a quotient of this finite group).
    For each $j \in J$, let $S_j = \{i \in \bN : s_j(g_{1,i},g_{1,i}^{-1}, \ldots, g_{K,i},g_{K,i}^{-1}) = \id\}$.
    Then by the definition of $\sim_{\cU}$, each $S_j$ is in $\cU$, and by property (ii) of ultrafilters, $S = \bigcap_j S_j$ is also in $\cU$, and is therefore nonempty.
    As in Remark~\ref{rmk:non-principal-is-cofinite}, since the ultrafilter is non-principal, $S$ must be infinite.
Fix an index $i$ in $S$ larger than $R$. Then $\{g_{1,i}, \ldots, g_{K,i}\}$ satisfy all the $s_j$, so they must generate a subgroup of size at most $R$.
    But by definition of the sequence $g_{i,j}$, $\{g_{1,i}, \ldots, g_{K,i}\}$ generates a subgroup of size at least $i$ for each $i$. Contradiction. The result follows.
\end{proof}

\section{A Myhill-Nerode analogue}
\label{sec:myhill-nerode}

In this section we develop an analogue of the Myhill-Nerode characterisation of regular languages~\cites{Myh57, Ner58} (see also~\cite{HU79}*{Section~3.4}).
The Myhill-Nerode characterisation relies on an equivalence relation on strings, which is defined via the operation of concatenation.
In~\cite{MNW18} we developed an idea that was inspired by Myhill-Nerode, but which used an equivalence relation on matroids defined via the (matroidal) operation of amalgamation.
In this section we generalise this technique to hypergraphs.
In order to establish an analogue of concatenation of strings, we develop the idea of coloured systems and coloured complements.
These can be glued together to form a hypergraph using an operation that we call a coloured sum.
This gives us the hypergraph analogue of concatenation that we use to construct a Myhill-Nerode-style equivalence.

\begin{definition}[$C$\dash coloured system]\label{def:tboundaried}
Let $C$ be a finite set.
A \emph{$C$\dash coloured system} is a pair $(U,c)$, where  $c$ is a function from $2^{U}$ to $C$.
\end{definition}

\begin{definition}[$C$\dash coloured complement]
\label{def:tcomplement}
Let $C$ be a finite set.
A \emph{$C$\dash coloured complement} is a finite set $V$ along with a function $d\colon 2^{V}\times C\to \{0,1\}$.
\end{definition}

\begin{definition}[$C$\dash coloured sum]
\label{def:tsum}
Let $C$ be a finite set and let $(U,c)$ be a  $C$\dash coloured system.
Let $(V,d)$ be a $C$\dash coloured complement, where $U$ and $V$ are disjoint.
The \emph{$C$\dash coloured sum} is
\[
(U\cup V, \{X\cup Y \colon X\subseteq U,\ Y\subseteq V,\ d(Y,c(X))=1\}).
\]
We denote this hypergraph by $(U,c)\boxplus (V,d)$.
\end{definition}

\begin{remark}
The models that we have in mind for coloured sums are versions of matroid sums.
Standard matroid operations, such as $1$\dash, $2$\dash, and $3$\dash sums can all be expressed as $C$\dash coloured sums, as can amalgams over finite sets.

To illustrate, we let $M$ and $N$ be matroids on the ground sets $U\cup \{p\}$ and $V\cup \{p\}$ respectively, where $U$, $V$, and $\{p\}$ are pairwise disjoint, and $p$ is not a loop or coloop in either matroid.
We define the function $c\colon 2^{U}\to [3]$ as follows.
If $X\subseteq U$ is dependent in $M$, then set $c(X)=1$.
If $X$ is independent and $p$ is contained in the closure of $X$, then set $c(X)=2$.
Otherwise, $X$ is independent and $p$ is not in the closure of $X$.
In this case, set $c(X)$ to be $3$.
Now $(U, c)$ is a $[3]$\dash coloured system.
Let $d\colon 2^{V}\times [3]\to \{0,1\}$ be the function such that $(Y,i)$ is taken to $1$ if $Y$ is an independent subset of $N$ and either $i=3$, or $i=2$ and $Y$ does not span $p$.
All other pairs are taken to $0$.
It is not difficult to check that the $[3]$\dash coloured sum $(U, c)\boxplus (V,d)$ is the hypergraph of independent sets of the $2$\dash sum of $M$ and $N$ along the basepoint $p$.
\end{remark}

\begin{definition}\label{def:invariance}
Let $\mathcal{M}$ be a family of hypergraphs and let $C$ be a finite set.
We define the relation $\sim_{\mathcal{M},C}$ on $C$\dash coloured systems.
Let $(U_{1},c_{1})$ and $(U_{2},c_{2})$ be two such systems.
Then
\[
(U_{1},c_{1})\sim_{\mathcal{M},C}(U_{2},c_{2})
\]
holds if, for every $C$\dash coloured complement $(V,d)$ such that $V\cap U_{1}=V\cap U_{2}=\emptyset$, we have 
\[
(U_{1},c_{1})\boxplus (V,d)\in \mathcal{M}
\leftrightarrow
(U_{2},c_{2})\boxplus (V,d)\in \mathcal{M}.
\]
It is clear that $\sim_{\mathcal{M},C}$ is an equivalence relation.
\end{definition}

\begin{definition}\label{def:Lambda-bound}
Let $s$, $t$, and $\delta$ be positive integers and let $\varphi$ be a $\delta$\dash confined \cmso\dash formula.
We define the integer $\Lambda_{\varphi}(s, t,\delta)$.
If $\varphi$ is $|Z_{i}|\equiv p\bmod{q}$, then we set $\Lambda_{\varphi}(s, t,\delta)$ to be $(\delta!)^{s}$.
If $\varphi$ is $\formula{hyp}(Z_{i})$, then $\Lambda_{\varphi}(s, t,\delta)$ is $t^{s}$, and if $\varphi$ is $Z_{i}\subseteq Z_{j}$, then $\Lambda_{\varphi}(s, t,\delta)$ is $2^{s^{2}}$.

We have defined $\Lambda_{\varphi}(s, t,\delta)$ when $\varphi$ is atomic.
Now assume $\varphi$ is $\neg\psi$.
In this case we set  $\Lambda_{\varphi}(s, t,\delta)$ to be  $\Lambda_{\psi}(s, t,\delta)$.
If $\varphi$ is $\psi_{1}\land \psi_{2}$, then we set  $\Lambda_{\varphi}(s, t,\delta)$ to be the product $\Lambda_{\psi_{1}}(s, t,\delta)\Lambda_{\psi_{2}}(s, t,\delta)$.
Finally, we assume that $\varphi$ is $\exists Z_{i}\psi$.
In this case we set $\Lambda_{\varphi}(s, t,\delta)$ to be
\[
2^{\Lambda_{\psi}(s, t,\delta)}.
\]
\end{definition}

\begin{lemma}\label{thm:sharp-Myhill-Nerode}
Let $t$, $\delta$, and $s$ be positive integers and let $\varphi$ be a $\delta$\dash confined \cmso\dash sentence with $s$ variables.
Let $\mathcal{M}$ be the class of hypergraphs that satisfy $\varphi$.
If $C$ is a set with cardinality $t$, then $\sim_{\mathcal{M},C}$ has at most $\Lambda_{\varphi}(s,t,\delta)$ equivalence classes.
\end{lemma}

We illustrate Lemma~\ref{thm:sharp-Myhill-Nerode} and build intuition by using it to strengthen Theorem~1.1 in~\cite{MNW18}.

\begin{theorem}\label{thm:non-def-rep}
The class of representable matroids is not \cmso\dash definable.
\end{theorem}

We note that this does indeed strengthen~\cite{MNW18}*{Theorem~1.1}, since that result applies only to \mso\dash definability.

\begin{proof}[Proof of \textup{Theorem~\ref{thm:non-def-rep}}]
Assume for a contradiction that $\varphi$ is a \cmso\dash sentence that is satisfied exactly by the hypergraphs that are representable matroids.
Let $s$ be the number of variables in $\varphi$ and let $\delta$ be the smallest positive integer such that $\varphi$ is $\delta$\dash confined.
Let $\mathcal{M}$ be the class of representable matroids.

For any prime number $p$, let $(U_{p},\mathcal{I}_{p})$ be a matroid isomorphic to the finite projective plane $\mathrm{PG}(2,p)$.
Set $c_{p}$ to be the function which takes every dependent subset of $U_{p}$ to $1$ and every independent subset to $2$.
Then $M_{p}=(U_{p},c_{p})$ is a $[2]$\dash coloured system.
Because there are infinitely many prime numbers~\cite{Euclid}*{Proposition~20}, Lemma~\ref{thm:sharp-Myhill-Nerode} tells us that there are distinct primes, $p_{1}$ and $p_{2}$, such that $M_{p_{1}}\sim_{\mathcal{M},[2]} M_{p_{2}}$ holds.

Let $N$ be a matroid $(V,\mathcal{I})$ that is isomorphic to $\mathrm{PG}(2,p_{1})$.
Let $d\colon 2^{V}\times [2] \to \{0,1\}$ be the function taking $(Y,i)$ to $1$ when $Y$ is independent in $N$ and $i=2$.
Other pairs are taken to $0$.
Then $M_{p_{i}}\boxplus (V,d)$ is isomorphic to the matroidal direct sum $\mathrm{PG}(2,p_{i})\oplus N$.
Therefore $M_{p_{1}}\boxplus (V,d)$ is representable over $\mathrm{GF}(p_{1})$ \cite{Oxl11}*{Proposition~4.2.11}.
On the other hand, both $\mathrm{PG}(2,p_{1})$ and $\mathrm{PG}(2,p_{2})$ are isomorphic to minors of $M_{p_{2}}\boxplus (V,d)$~\cite{Oxl11}*{4.2.19}.
It follows from \cite{Oxl11}*{Proposition~3.2.4} and \cite{Aig97}*{Proposition~7.3} that if $M_{p_{2}}\boxplus (V,d)$ is representable over a field, then that field must simultaneously have subfields isomorphic to $\mathrm{GF}(p_{1})$ and $\mathrm{GF}(p_{2})$, an impossibility.
To summarise, $M_{p_{1}}\boxplus (V,d)$ is representable and $M_{p_{2}}\boxplus (V,d)$ is not.
Therefore $M_{p_{1}}\nsim_{\mathcal{M},[2]} M_{p_{2}}$, which is a contradiction.
\end{proof}

The proof of Lemma~\ref{thm:sharp-Myhill-Nerode} involves several technical definitions.
However, the basic idea is not too complicated.
A \emph{cleft} (Definition~\ref{def:cleft-definition}) is a certificate that two $C$\dash coloured set systems are not equivalent under $\sim_{\mathcal{M},C}$.
Definition~\ref{def:registries} introduces the idea of a piece of information carried by each $C$\dash coloured system.
The exact form of this piece of information will depend on the structure of $\varphi$, but the important point is that there are at most $\Lambda_{\varphi}(s,t,\delta)$ values that this information can take (Proposition~\ref{prop:num-registries}).
Furthermore, if two $C$\dash coloured systems carry the same piece of information, then there can be no cleft that divides them (Corollary~\ref{cor:reg-equivalance}).
Lemma~\ref{thm:sharp-Myhill-Nerode} follows from these steps.

\begin{definition}[Cleft]\label{def:cleft-definition}
Let $C$ be a finite set and for $i=1,2$, let $M_{i}=(U_{i},c_{i})$ be a $C$\dash coloured system.
Let $\varphi$ be a \cmso\dash formula and for $i=1,2$, let $\sigma_{i}$ be a function from $\free(\varphi)$ to $2^{U_{i}}$.
A \emph{$\varphi$\dash cleft} for $(M_{1},\sigma_{1})$ and $(M_{2},\sigma_{2})$ consists of a $C$\dash coloured complement $(V,d)$ and a function $\tau\colon \free(\varphi)\to 2^{V}$ such that $V\cap U_{1}=V\cap U_{2}=\emptyset$ and the following holds:
for $i=1,2$, if we define $N_{i}$ to be
\[
M_{i}\boxplus (V,d)
\]
and let $\theta_{i}$ be the function taking each $Z_{i}\in \free(\varphi)$ to $\sigma_{i}(Z_{i})\cup \tau(Z_{i})$, then $\varphi$ is satisfied by exactly one of $(N_{1},\theta_{1})$ and $(N_{2},\theta_{2})$.
\end{definition}

So when $\varphi$ is a sentence that defines the class $\mathcal{M}$, there is a $\varphi$\dash cleft exactly when $M_{1}$ and $M_{2}$ are not equivalent under $\sim_{\mathcal{M},C}$.

\begin{definition}\label{def:registries}
Let $S\subseteq \{Z_{1},Z_{2},\ldots\}$ be a finite set of variables, let $C$ be a finite set, and let $\delta$ be a positive integer.
We are going to define a function $R_{S,C,\delta}$ which takes as input any triple $(M,\varphi,\sigma)$, where:
\begin{enumerate}[label = $\bullet$]
\item $M = (U, c)$ is a $C$\dash coloured system,
\item $\varphi$ is a $\delta$\dash confined \cmso\dash formula such that $\var(\varphi)\subseteq S$, and
\item $\sigma$ is a function from $S-\bound(\varphi)$ to $2^{U}$.
\end{enumerate}
We define the output of $R_{S,C,\delta}$ recursively.
Let $T$ stand for $S-\bound(\varphi)$, so that $\sigma$ is a function from $T$ to $2^{U}$.
If $\varphi$ is an atomic formula, then $\bound(\varphi)=\emptyset$ and $T=S$.
In this case, $R_{S,C,\delta}(M,\varphi,\sigma)$ is defined as follows.
\begin{enumerate}[label = \textup{(\roman*)}]
\item If $\varphi$ is $|Z_{i}|\equiv p\bmod{q}$ for some $Z_{i}\in S$, then $R_{S,C,\delta}(M,\varphi,\sigma)$ is the function $r_{1}\colon S\to \{0,1,\ldots,\delta!-1\}$, where $r_{1}(Z_{i})$ is the residue of $|\sigma(Z_{i})|$ modulo $\delta!$ for each $Z_{i}\in S$.
\item If $\varphi$ is $\formula{hyp}(Z_{i})$ for some $Z_{i}\in S$, then $R_{S,C,\delta}(M,\varphi,\sigma)$ is the function $r_{2}\colon S\to C$, where $r_{2}(Z_{i})$ is  $c(\sigma(Z_{i}))$ for each $Z_{i}\in S$.
\item If $\varphi$ is $Z_{i}\subseteq Z_{j}$ for some $Z_{i},Z_{j}\subseteq S$, then $R_{S,C,\delta}(M,\varphi,\sigma)$ is the function $r_{3}\colon S\times S\to\{0,1\}$, where $r_{3}(Z_{i},Z_{j})=1$ if and only if $\sigma(Z_{i})\subseteq \sigma(Z_{j})$.
\end{enumerate}
We have now defined $R_{S,C,\delta}(M,\varphi,\sigma)$ in the case that $\varphi$ is atomic.
Assume now that $\varphi$ is not atomic.
If $\varphi$ is $\neg\psi$ then $\bound(\varphi)=\bound(\psi)$ and we declare
$R_{S,C,\delta}(M,\varphi,\sigma)$ to be equal to $R_{S,C,\delta}(M,\psi,\sigma)$.
Next assume that $\varphi=\psi_{1}\land\psi_{2}$.
Note that $\bound(\varphi)=\bound(\psi_{1})\cup\bound(\psi_{2})$ and
\[\bound(\psi_{i})\cap\free(\psi_{3-i})=\emptyset\]
for $i=1,2$.
Let $\sigma_{i}$ be the function
\[
\sigma\cup\{(Z_{s},\emptyset)\colon Z_{s}\in \bound(\psi_{3-i})-\bound(\psi_{i})\}.
\]
Thus $\sigma_{i}$ is a function from $S-\bound(\psi_{i})$ to subsets of $U$.
We set $R_{S,C,\delta}(M,\varphi,\sigma)$ to be the ordered pair
\[
(R_{S,C,\delta}(M,\psi_{1},\sigma_{1}),R_{S,C,\delta}(M,\psi_{2},\sigma_{2})).
\]
Finally we assume that $\varphi=\exists Z_{i}\psi$.
In this case
\[\bound(\psi) = \bound(\varphi)+Z_{i}.\]
We declare $R_{S,C,\delta}(M,\varphi,\sigma)$ to be the set
\[
\{R_{S,C,\delta}(M,\psi,\sigma + (Z_{i},X))\colon X\subseteq U\}.
\]
That is, we consider every extension of $\sigma$ by an ordered pair of the form $(Z_{i},X)$.
We then let $R_{S,C,\delta}(M,\varphi,\sigma)$ be the set of all outputs produced by $R_{S,C,\delta}$ operating on these extensions.
\end{definition}

We observe that $R_{S,C,\delta}(M,\varphi,\sigma)$ is either a function with $S$ or $S\times S$ as its domain, or it is a ordered pair, or it is a set.
Thus two outputs $R_{S,C,\delta}(M_{1},\varphi,\sigma_{1})$ and $R_{S,C,\delta}(M_{2},\varphi,\sigma_{2})$ are equal if they are equal as functions, pairs, or sets.

\begin{definition}[Sympathetic]\label{def:compatability}
Let $S\subseteq \{Z_{1},Z_{2},\ldots\}$ be a finite set of variables, let $C$ be a finite set, and let $\delta$ be a positive integer.
Let $M=(U,c)$ be a $C$\dash coloured system.
Let $\varphi$ be a $\delta$\dash confined \cmso\dash formula satisfying $\var(\varphi)\subseteq S$ and let $\sigma$ be a function from $S-\bound(\varphi)$ to $2^{U}$.
Let $\Pi=(V,d)$ be a $C$\dash coloured complement, where we assume that $U$ and $V$ are disjoint.
Let $\tau$ be a function from $\free(\varphi)$ to $2^{V}$.
We are going to define what it means for $R_{S,C,\delta}(M,\varphi,\sigma)$ and $(\Pi,\tau)$ to be \emph{sympathetic}.

If $\varphi$ is atomic then $R_{S,C,\delta}(M,\varphi,\sigma)$ is a function.
First assume that $\varphi$ is $|Z_{i}|\equiv p\bmod{q}$, so that $R_{S,C,\delta}(M,\varphi,\sigma)$ is the function $r_{1}$ from $S$ to $\{0,1,\ldots, \delta!-1\}$.
In this case we say that $R_{S,C,\delta}(M,\varphi,\sigma)$ and $(\Pi,\tau)$ are sympathetic if
\[r_{1}(Z_{i})+|\tau(Z_{i})|\ \text{is congruent to $p$ modulo $q$}.\]
Next assume that $\varphi$ is $\formula{hyp}(Z_{i})$, so $R_{S,C,\delta}(M,\varphi,\sigma)$ is the function $r_{2}\colon S\to C$.
We declare that $R_{S,C,\delta}(M,\varphi,\sigma)$ and $(\Pi,\tau)$ are sympathetic if
\[d(\tau(Z_{i}),r_{2}(Z_{i}))=1.\]
In the next case we assume $\varphi$ is $Z_{i}\subseteq Z_{j}$, so that $R_{S,C,\delta}(M,\varphi,\sigma)$ is the function $r_{3}\colon S\times S\to \{0,1\}$.
Then $R_{S,C,\delta}(M,\varphi,\sigma)$ and $(\Pi,\tau)$ are sympathetic if $r_{3}(Z_{i},Z_{j})=1$ and $\tau(Z_{i})\subseteq \tau(Z_{j})$.

We will now assume that $\varphi$ is not atomic.
Assume that $\varphi$ is $\neg\psi$.
Then $R_{S,C,\delta}(M,\varphi,\sigma)$ and $(\Pi,\tau)$ are sympathetic if and only if $R_{S,C,\delta}(M,\psi,\sigma)$ and $(\Pi,\tau)$ are \emph{not} sympathetic.

Next assume that $\varphi$ is $\psi_{1}\land\psi_{2}$.
We define $\sigma_{i}$ exactly as in Definition~\ref{def:registries}, so that it is a function from $S-\bound(\psi_{i})$ to $2^{U}$.
For $i=1,2$, let $\tau_{i}$ be the restriction of $\tau$ to $\free(\psi_{i})$.
Now $R_{S,C,\delta}(M,\varphi,\sigma)$ and $(\Pi,\tau)$ are sympathetic if and only if $R_{S,C,\delta}(M,\psi_{i},\sigma_{i})$ and $(\Pi,\tau_{i})$ are sympathetic, for $i=1,2$.

Finally, assume that $\varphi$ is $\exists Z_{i}\psi$.
Then $R_{S,C,\delta}(M,\varphi,\sigma)$ and $(\Pi,\tau)$ are sympathetic if and only if there exist subsets $X\subseteq U$ and $Y\subseteq V$ such that $R_{S,C,\delta}(M,\psi,\sigma+(Z_{i},X))$ and $(\Pi,\tau+(Z_{i},Y))$ are sympathetic.
\end{definition}

\begin{proposition}\label{prop:compatibility}
Let $C$ be a finite set and let $\delta$ be a positive integer.
Let $M=(U, c)$ be a $C$\dash coloured system.
Let $\varphi$ be a $\delta$\dash confined \cmso\dash formula.
Set $S$ to be $\var(\varphi)$ and let $\sigma$ be a function from $\free(\varphi)$ to $2^{U}$.
Let $\Pi=(V,d)$ be a $C$\dash coloured complement where we assume $U$ and $V$ are disjoint, and let $\tau$ be a function from $\free(\varphi)$ to $2^{V}$.
Set $E$ to be $U\cup V$ and define $\theta\colon \free(\varphi)\to 2^{E}$ so that $\theta(Z_{s})=\sigma(Z_{s})\cup \tau(Z_{s})$ for every free variable $Z_{s}$.
Let $N$ be the hypergraph $M\boxplus \Pi$.
Then $\varphi$ is satisfied by $(N,\theta)$ if and only if $R_{S,C,\delta}(M,\varphi,\sigma)$ and $(\Pi,\tau)$ are sympathetic.
\end{proposition}

\begin{proof}
The proof is by induction on the number of steps required to construct $\varphi$.
First assume that $\varphi$ is atomic.
Then $R_{S,C,\delta}(M,\varphi,\sigma)$ is one of the functions $r_{1}$, $r_{2}$, or $r_{3}$.
Assume that $\varphi$ is $|Z_{i}|\equiv p\bmod{q}$.
Then $\varphi$ is satisfied by $(N,\theta)$ if and only if
\[|\theta(Z_{i})|=|\sigma(Z_{i})|+|\tau(Z_{i})|\]
is congruent to $p$ modulo $q$.
Note that $|\sigma(Z_{i})|$ can be expressed as $r_{1}(Z_{i})+\kappa\delta!$ for some integer $\kappa$, by the definition of the $r_{1}$ function.
Since $\varphi$ is $\delta$\dash confined it follows that $q\leq \delta$ and therefore $q$ divides $\delta!$.
It follows that $|\sigma(Z_{i})|+|\tau(Z_{i})|$ is congruent to $p$ modulo $q$ if and only if $r_{1}(Z_{i})+|\tau(Z_{i})|$ is congruent to $p$ modulo $q$.
This is true if and only if $R_{S,C,\delta}(M,\varphi,\sigma)$ and $(\Pi,\tau)$ are sympathetic.
Therefore we are done in the case that $\varphi$ is $|Z_{i}|\equiv p\bmod{q}$.

Assume that $\varphi$ is the formula $\formula{hyp}(Z_{i})$.
Then $\varphi$ is satisfied by $(N,\theta)$ if and only if $\theta(Z_{i})=\sigma(Z_{i})\cup \tau(Z_{i})$ is a hyperedge of $M\boxplus \Pi$.
This set of hyperedges is
\[
\{X\cup Y\colon X\subseteq U,\ Y\subseteq V, d(Y,c(X))=1\},
\]
so $\varphi$ is satisfied by $(N,\theta)$ if and only if
\[1=d(\tau(Z_{i}),c(\sigma(Z_{i})))=d(\tau(Z_{i}),r_{2}(Z_{i})).\]
This is exactly what it means for $R_{S,C,\delta}(M,\varphi,\sigma)$ and $(\Pi,\tau)$ to be sympathetic.

Next we assume that $\varphi$ is $Z_{i}\subseteq Z_{j}$.
Then $\varphi$ is satisfied by $(N,\theta)$ if and only if $\sigma(Z_{i})\subseteq \sigma(Z_{j})$ and $\tau(Z_{i})\subseteq \tau(Z_{j})$.
This is true if and only if $r_{3}(Z_{i},Z_{j})=1$ and $\tau(Z_{i})\subseteq \tau(Z_{j})$, which in turn is true if and only if $R_{S,C,\delta}(M,\varphi,\sigma)$ and $(\Pi,\tau)$ are sympathetic.

We are now done with the case that $\varphi$ is atomic.
Therefore we consider the case that $\varphi$ is $\neg\psi$.
Hence $\varphi$ is satisfied by $(N,\theta)$ if and only if $\psi$ is not satisfied by $(N,\theta)$.
By induction, this is the case if and only if $R_{S,C,\delta}(M,\psi,\sigma)$ and $(\Pi,\tau)$ are not sympathetic, which is precisely the condition for $R_{S,C,\delta}(M,\varphi,\sigma)$ and $(\Pi,\tau)$ to be sympathetic.

The next case is when $\varphi$ is $\psi_{1}\land\psi_{2}$.
For $i=1,2$, we let $\tau_{i}$ (respectively $\theta_{i}$) be the restriction of $\tau$ (respectively $\theta$) to $\free(\psi_{i})$.
Define $\sigma_{i}$ to be
\[\sigma\cup \{(Z_{s},\emptyset)\colon Z_{s}\in\bound(\psi_{3-i})-\bound(\psi_{i})\}.\]
Now $\theta_{i}$ takes each $Z_{s}\in\free(\psi)$ to $\sigma_{i}(Z_{s})\cup\tau_{i}(Z_{s})$.
We see that $\varphi$ is satisfied by $(N,\theta)$ if and only if $\psi_{i}$ is satisfied by $(N,\theta_{i})$ for $i=1,2$.
By induction, this is true if and only if $R_{S,C,\delta}(M,\psi_{i},\sigma_{i})$ and $(\Pi,\tau_{i})$ are sympathetic for $i=1,2$.
This is true if and only if $R_{S,C,\delta}(M,\varphi,\sigma)$ and $(\Pi,\tau)$ are sympathetic.

Finally, we assume that $\varphi$ is $\exists Z_{s}\psi$.
We argue as follows.
\begin{align*}
&\quad \varphi\ \text{is satisfied by}\ (N,\theta)\\
\leftrightarrow&\quad \text{there exists}\ D\subseteq E\ \text{such that}\ \psi\ \text{is satisfied by}\ (N,\theta+(Z_{s},D))\\
\leftrightarrow&\quad \text{there exists}\ D\subseteq E\ \text{such that}\ \psi\ \text{is satisfied by}\ (N,\theta')\ \text{where we}\\
&\qquad\qquad\text{define}\ \sigma'=\sigma+(Z_{s},D\cap U)\ \text{and}\ \tau'=\tau+(Z_{s},D\cap V)\\
&\qquad\qquad\text{and $\theta'$ maps each}\ Z_{t}\in \free(\psi)\ \text{to}\ \sigma'(Z_{t})\cup \tau'(Z_{t})\\
\leftrightarrow&\quad \text{there exist}\ X\subseteq U\ \text{and}\ Y\subseteq V\ \text{such that}\ \psi\ \text{is satisfied by}\ (N,\theta')\\
&\qquad\qquad \text{where we define}\ \sigma'=\sigma+(Z_{s},X)\ \text{and}\ \tau'=\tau+(Z_{s},Y)\\
&\qquad\qquad\ \text{and $\theta'$ maps each}\ Z_{t}\in \free(\psi)\ \text{to}\ \sigma'(Z_{t})\cup \tau'(Z_{t})\\
\leftrightarrow&\quad \text{there exist}\ X\subseteq U\ \text{and}\ Y\subseteq V\ \text{such that}\\
&\qquad\qquad R_{S,C,\delta}(M,\psi,\sigma+(Z_{s},X))\ \text{and}\ (\Pi,\tau+(Z_{s},Y))\\
&\qquad\qquad\text{are sympathetic}\\
\leftrightarrow&\quad R_{S,C,\delta}(M,\varphi,\sigma)\ \text{and}\ (\Pi,\tau)\ \text{are sympathetic}
\end{align*}
Now the proof of Proposition~\ref{prop:compatibility} is complete.
\end{proof}

\begin{corollary}\label{cor:reg-equivalance}
Let $C$ be a finite set and let $\delta$ be a positive integer.
For $i=1,2$ let $M_{i}=(U_{i},c_{i})$ be a $C$\dash coloured system.
Let $\varphi$ be a $\delta$\dash confined \cmso\dash formula and for $i=1,2$ let $\sigma_{i}$ be a function from $\free(\varphi)$ to $2^{U_{i}}$.
Let $S$ be $\var(\varphi)$.
If there is a $\varphi$\dash cleft for $(M_{1},\sigma_{1})$ and $(M_{2},\sigma_{2})$, then
\[R_{S,C,\delta}(M_{1},\varphi,\sigma_{1})\ne R_{S,C,\delta}(M_{2},\varphi,\sigma_{2}).\]
\end{corollary}

\begin{proof}
Assume for a contradiction that
\[R_{S,C,\delta}(M_{1},\varphi,\sigma_{1})= R_{S,C,\delta}(M_{2},\varphi,\sigma_{2})\]
and yet we have a $\varphi$\dash cleft for $(M_{1},\sigma_{1})$ and $(M_{2},\sigma_{2})$.
Let this cleft consist of the $C$\dash coloured complement $\Pi=(V,d)$ and the function $\tau\colon \free(\varphi)\to 2^{V}$.
For $i=1,2$, let $\theta_{i}\colon \free(\varphi)\to 2^{U\cup V}$ be the function taking each $Z_{s}\in\free(\varphi)$ to $\sigma_{i}(Z_{s})\cup \tau(Z_{s})$.
Let $N_{i}$ be $M_{i}\boxplus (V,d)$.
Definition~\ref{def:cleft-definition} means that $\varphi$ is satisfied by exactly one of $(N_{1},\theta_{1})$ and $(N_{2},\theta_{2})$.

Proposition~\ref{prop:compatibility} says that $\varphi$ is satisfied by $(N_{1},\theta_{1})$ if and only $R_{S,C,\delta}(M_{1},\varphi,\sigma_{1})$ and $(\Pi,\tau)$ are sympathetic.
As
\[R_{S,C,\delta}(M_{1},\varphi,\sigma_{1})= R_{S,C,\delta}(M_{2},\varphi,\sigma_{2})\]
holds, this is the case if and only if $R_{S,C,\delta}(M_{2},\varphi,\sigma_{2})$ and $(\Pi,\tau)$ are sympathetic, which holds if and only if $\varphi$ is satisfied by $(N_{2},\theta_{2})$.
Now we have contradicted our earlier conclusion.
\end{proof}

Recall that the integer $\Lambda_{\varphi}(s,t,\delta)$ was described in Definition~\ref{def:Lambda-bound}.

\begin{proposition}\label{prop:num-registries}
Let $S\subseteq \{Z_{1},Z_{2},\ldots\}$ be a finite set of variables. Let $C$ be a finite set, and let $\delta$ be a positive integer.
Let $\varphi$ be a $\delta$\dash confined \cmso\dash formula such that $\var(\varphi)\subseteq S$.
As $M$ ranges over all $C$\dash coloured systems $(U,c)$, and $\sigma$ ranges over all functions from $S-\bound(\varphi)$ to $2^{U}$, the number of values taken by $R_{S,C,\delta}(M,\varphi,\sigma)$ is no greater than $\Lambda_{\varphi}(|S|,|C|,\delta)$.
\end{proposition}

\begin{proof}
The proof is by induction on the number of steps needed to construct $\varphi$.
Assume $\varphi$ is $|Z_{i}|\equiv p\bmod{q}$.
Then $R_{S,C,\delta}(M,\varphi,\sigma)$ is a function $r_{1}$ from $S$ to $\{0,1,\ldots, \delta!-1\}$.
The number of such functions is at most $(\delta!)^{|S|} = \Lambda_{\varphi}(|S|,|C|,\delta)$, so the result holds.
Similarly, if $\varphi$ is $\formula{hyp}(Z_{i})$, then $R_{S,C,\delta}(M,\varphi,\sigma)$ is a function $r_{2}$ from $S$ to $C$, and there are at most $|C|^{|S|} = \Lambda_{\varphi}(|S|,|C|,\delta)$ such functions.
If $\varphi$ is $Z_{i}\subseteq Z_{j}$, then $R_{S,C,\delta}(M,\varphi,\sigma)$ is the function $r_{3}$ from $S\times S$ to $\{0,1\}$, and there are at most $2^{|S|^{2}}$ such functions.
So the result holds in the case that $\varphi$ is atomic.

If $\varphi$ is $\neg\psi$, then $\bound(\varphi)=\bound(\psi)$ and $R_{S,C,\delta}(M,\varphi,\sigma)=R_{S,C,\delta}(M,\psi,\sigma)$ for every choice of $M$ and $\sigma$.
Furthermore, $\Lambda_{\varphi}(|S|,|C|,\delta) = \Lambda_{\psi}(|S|,|C|,\delta)$.
Thus the result holds by induction.
Therefore we will assume that $\varphi$ is $\psi_{1}\land\psi_{2}$.
We recall that $\bound(\varphi)=\bound(\psi_{1})\cup\bound(\psi_{2})$ and that no variable is bound in exactly one of $\psi_{1}$ and $\psi_{2}$.
Recall also that when $\sigma$ is a function from $S-\bound(\varphi)$ to $2^{U}$, then
\[
\sigma_{i}=\sigma\cup\{(Z_{s},\emptyset)\colon Z_{s}\in\bound(\psi_{3-i})-\bound(\psi_{i})\}
\]
for $i=1,2$.
Now $R_{S,C,\delta}(M,\varphi,\sigma)$ is the ordered pair
\[
(R_{S,C,\delta}(M,\psi_{1},\sigma_{1}),R_{S,C,\delta}(M,\psi_{2},\sigma_{2})).
\]
So the number of values taken by $R_{S,C,\delta}(M,\varphi,\sigma)$ is at most the product of the numbers of values taken by $R_{S,C,\delta}(M,\psi_{1},\sigma_{1})$ and $R_{S,C,\delta}(M,\psi_{2},\sigma_{2})$.
By induction, this is at most $\Lambda_{\psi_{1}}(|S|,|C|,\delta)\Lambda_{\psi_{2}}(|S|,|C|,\delta)$, which is equal to $\Lambda_{\varphi}(|S|,|C|,\delta)$.

Now we must assume that $\varphi$ is $\exists Z_{s}\psi$.
Then $R_{S,C,\delta}(M,\varphi,\sigma)$ is a set of outputs of the form $R_{S,C,\delta}(M,\psi,\sigma+(Z_{s},X))$.
By induction the number of such subsets is at most
\[2^{\Lambda_{\psi}(|S|,|C|,\delta)} = \Lambda_{\varphi}(|S|,|C|,\delta)\]
so the proof is complete.
\end{proof}

\begin{proof}[Proof of \textup{Lemma~\ref{thm:sharp-Myhill-Nerode}}]
Let $s$, $t$, and $\delta$, be positive integers, and let $\varphi$ be a $\delta$\dash confined $\cmso$\dash sentence with $s$ variables.
Let $S$ be the set of variables in $\varphi$.
Let $\mathcal{M}$ be the class of hypergraphs that satisfy $\varphi$ and let $C$ be a set with cardinality $t$.
We claim that the number of equivalence classes of $\sim_{\mathcal{M}, C}$ is no greater than $\Lambda_{\varphi}(s, t,\delta)$.
Let $M_{1},M_{2},\ldots$ be representatives of these equivalence classes, so that each $M_{i}$ is a $C$\dash coloured system.
Since $M_{i}\nsim_{\mathcal{M},C} M_{j}$ when $i\ne j$, we can find a $C$\dash coloured complement $(V,d)$ such that exactly one of $M_{i}\boxplus (V,d)$ and $M_{j}\boxplus (V,d)$ is in $\mathcal{M}$.
Because $\free(\varphi)$ is empty, this complement is a $\varphi$\dash cleft for $M_{i}$ and $M_{j}$.
Corollary~\ref{cor:reg-equivalance} now implies that $R_{S, C, \delta}(M_{i}, \varphi, \emptyset) \ne R_{S, C, \delta}(M_{j}, \varphi, \emptyset)$ when $i\ne j$.
Therefore the number of equivalence classes under $\sim_{\mathcal{M}, C}$ is no greater than the number of values taken by $R_{S,C,\delta}(M,\varphi,\emptyset)$ as $M$ ranges over all $C$\dash coloured systems.
This is at most $\Lambda_{\varphi}(s, t,\delta)$, by Proposition~\ref{prop:num-registries}, so we are done.
\end{proof}

\section{Amalgams}
\label{sec:amalgams}

Let $M_{1}$ and $M_{2}$ be matroids with ground sets
$E_{1}$ and $E_{2}$, rank functions $r_{1}$ and $r_{2}$, and
closure operators $\cl_{1}$ and $\cl_{2}$.
Let $\ell$ be $E_{1}\cap E_{2}$.
We assume that $M_{1}|\ell=M_{2}|\ell$ and we denote this shared restriction by $N$.
If $M$ is a matroid on the ground set $E_{1}\cup E_{2}$ and $M|E_{i}=M_{i}$ for $i=1,2$, then we say $M$ is an \emph{amalgam} of $M_{1}$ and $M_{2}$.

A matroid is \emph{modular} if
$r(F)+r(F')=r(F\cap F')+r(F\cup F')$ whenever $F$ and $F'$ are flats.
Assume that $N$ is a modular matroid.
Then~\cite{Oxl11}*{Theorem~11.4.10} tells us that we can obtain an amalgam of $M_{1}$ and $M_{2}$ by setting the rank of any subset $X\subseteq E_{1}\cup E_{2}$ to be
\begin{equation}
\label{eqn1}
\min\{r_{1}(Y\cap E_{1})+r_{2}(Y\cap E_{2})-r_{N}(Y\cap \ell)\ \colon\
X\subseteq Y\subseteq E_{1}\cup E_{2}\}    
\end{equation}
The resulting matroid is the \emph{proper amalgam} of $M_{1}$ and $M_{2}$.
We refer to the set $\ell$ as the \emph{amalgam base} and
we denote the proper amalgam by \amal{M_{1}}{M_{2}}[\ell], or \amal{M_{1}}{M_{2}} if the amalgam base is clear from the context.
We can easily check that every rank\dash $2$ matroid is modular.
The following result is slightly more powerful than Proposition~4.1 in~\cite{MNW18} because it allows the matroids to be non-simple.

\begin{proposition}\label{prop:amalgams}
For $i=1,2$, let $M_{i}$ be a matroid with ground set $E_{i}$, rank function $r_{i}$, and closure operator $\cl_{i}$.
Let $\ell=E_{1}\cap E_{2}$, where $M_{1}|\ell=M_{2}|\ell$ and $r_{1}(\ell)=2$.
Let $X$ be a subset of $E_{1}\cup E_{2}$.
If $X\cap E_{1}$ is dependent in $M_{1}$ or if $X\cap E_{2}$ is dependent in $M_{2}$, then $X$ is dependent in \amal{M_{1}}{M_{2}}.
If $X\cap E_{1}$ is independent in $M_{1}$ and $X\cap E_{2}$ is independent in $M_{2}$, then $X$ is dependent in \amal{M_{1}}{M_{2}}\ if and only if
\begin{enumerate}[label=\textup{(\roman*)}]
\item $\ell\subseteq \cl_{1}(X\cap E_{1})$ and $X-E_{1}$ is not skew with $\ell$ in $M_{2}$,
\item $\ell\subseteq \cl_{2}(X\cap E_{2})$ and $X-E_{2}$ is not skew with $\ell$ in $M_{1}$, or
\item $\cl_{1}(X-E_{2}) \cap\cl_{2}(X-E_{1})$ contains a non-loop element.
\end{enumerate}
\end{proposition}

\begin{proof}
We use $N$ to denote $M_{1}|\ell = M_{2}|\ell$.
Assume $X\cap E_{1}$ is dependent in $M_{1}$.
Then $X\cap E_{1}$ is dependent in \amal{M_{1}}{M_{2}}
since $\amal{M_{1}}{M_{2}}|E_{1}=M_{1}$.
Thus $X$ is dependent in \amal{M_{1}}{M_{2}}.
By symmetry we conclude that if $X\cap E_{1}$ is dependent in $M_{1}$ or if $X\cap E_{2}$ is dependent in $M_{2}$, then $X$ is dependent in \amal{M_{1}}{M_{2}}.
Henceforth we assume $X\cap E_{1}$ is independent in $M_{1}$ and $X\cap E_{2}$ is independent in $M_{2}$.

Assume statement (i) holds.
Then
\[
r_{2}((X-E_{1})\cup \ell) < r_{2}(X-E_{1})+ r_{2}(\ell) = r_{2}(X-E_{1})+2.
\]
Let $Y$ be $X\cup \ell$, so that
\[r_{1}(Y\cap E_{1}) = r_{1}((X\cap E_{1})\cup \ell) = r_{1}(X\cap E_{1})\]
because $\ell$ is in the closure of $X\cap E_{1}$.
Now
\begin{align*}
|X|&= |X \cap E_{1}|+|X-E_{1}|\\
&= r_{1}(X\cap E_{1})+r_{2}(X-E_{1})\\
&> r_{1}(X\cap E_{1})+r_{2}((X-E_{1})\cup\ell)-2\\
&= r_{1}(Y\cap E_{1})+r_{2}(Y\cap E_{2})-r_{N}(Y\cap \ell),
\end{align*}
so~\eqref{eqn1} implies the rank of $X$ in \amal{M_{1}}{M_{2}} is less than $|X|$ and therefore $X$ is dependent in \amal{M_{1}}{M_{2}}.
By symmetric arguments, we see that if (i) or (ii) holds, then $X$ is dependent in \amal{M_{1}}{M_{2}}.

Next we assume that (iii) holds.
Let $y$ be a non-loop element in
\[\cl_{1}(X-E_{2}) \cap \cl_{2}(X-E_{1})\]
so that $y$ is necessarily in $\ell$.
As $y$ is in $\cl_{1}(X-E_{2})$ there is a circuit of $M_{1}$ contained in $(X-E_{2})\cup y$ that contains $y$.
Since $X\cap E_{1}$ contains no such circuit, it follows that $y$ is not in $X$.

Assume $X\cap \ell$ is non-empty and let $x$ be an arbitrary element of this set.
Assume that $\{x,y\}$ is dependent in $M_{1}$.
Since $X\cap E_{1}$ is independent in $M_{1}$ it follows that $x$ is not a loop.
Therefore $\{x,y\}$ is a circuit.
There is a circuit contained in $(X-E_{1})\cup y$ that contains $y$.
Performing circuit elimination on this circuit and $\{x,y\}$ produces a circuit of $M_{1}$ contained in $X\cap E_{1}$.
This is a contradiction, so $\{x,y\}$ is independent in $M_{1}$.

Assume $X\cap \ell$ contains distinct elements, $x$ and $x'$.
As $X\cap E_{1}$ is independent in $M_{1}$ we see that $\{x,x'\}$ is independent in $M_{1}$.
The previous paragraph shows that $\{x,y\}$ and $\{x',y\}$ are independent in $M_{1}$, so $r(N)=2$ implies that $\{x,x',y\}$ is a circuit.
By performing circuit elimination on $\{x,x',y\}$ and a circuit
contained in $(X-E_{2})\cup y$ that contains $y$, we
obtain a circuit of $M_{1}$ contained in $X\cap E_{1}$.
This contradiction means that $|X\cap \ell|\in\{0,1\}$.

Let $Y$ be $X\cup y$.
Because $y$ is in $\cl_{1}(X-E_{2})$ it follows that $r_{1}(X\cap E_{1}) = r_{1}(Y\cap E_{1})$.
Similarly, $r_{2}(X\cap E_{2}) = r_{2}(Y\cap E_{2})$.
If $X\cap \ell = \emptyset$ then $0=|X\cap \ell| = r_{N}(Y\cap \ell) - 1$.
Now assume that $X\cap \ell = \{x\}$.
Since $\{x,y\}$ is independent in $N$, we have $r_{N}(Y\cap \ell) = 2$, so $|X\cap \ell| = r_{N}(Y\cap \ell) - 1$ holds in either case.
Now we see that
\begin{align*}
|X|&=|X\cap E_{1}|+|X\cap E_{2}|-|X\cap \ell |\\
&=r_{1}(X\cap E_{1})+r_{2}(X\cap E_{2})-|X\cap \ell |\\
&=r_{1}(Y\cap E_{1})+r_{2}(Y\cap E_{2})-(r_{N}(Y\cap \ell)-1)\\
&>r_{1}(Y\cap E_{1})+r_{2}(Y\cap E_{2})-r_{N}(Y\cap \ell).
\end{align*}
Again we see that $X$ is dependent in \amal{M_{1}}{M_{2}}, and
this completes the proof of the `if' direction.

For the `only if' direction, we assume that $X$ is
dependent in \amal{M_{1}}{M_{2}}.
As $X\cap E_{1}$ is independent in $M_{1}$ and
$X\cap E_{2}$ is independent in $M_{2}$, it follows that
$X$ is contained in neither $E_{1}$ nor $E_{2}$.
There is some set $Y$ such that $X\subseteq Y\subseteq E_{1}\cup E_{2}$
and $|X|>r_{1}(Y\cap E_{1})+r_{2}(Y\cap E_{2})-r_{N}(Y\cap \ell)$.
Assume that amongst all such sets, $Y$ has been chosen so that
it is as small as possible.
If $y$ is an element in $Y-(X\cup E_{2})$, then we could replace $Y$ with $Y-y$.
Therefore no such element exists.
By symmetry it follows that $Y-X\subseteq \ell$.

If $Y$ contains a loop element $y$, then $y$ is in $Y-X$, since $X\cap E_{1}$ and $X\cap E_{2}$ are independent in $M_{1}$ and $M_{2}$ respectively.
But in this case we could replace $Y$ with $Y-y$, so $Y$ contains no loops.

If $Y=X$, then $Y\cap E_{1}$ is independent in $M_{1}$ and
$Y\cap E_{2}$ is independent in $M_{2}$, so
$|X|>r_{1}(Y\cap E_{1})+r_{2}(Y\cap E_{2})-r_{N}(Y\cap \ell)=|Y|=|X|$.
This contradiction means that $Y-X$ is non-empty.

\begin{claim}
\label{clm:amalgam-clm1}
If $y$ is in $Y-X$, then
\[y\in \cl_{1}((Y-y)\cap E_{1}) \cap \cl_{2}((Y-y)\cap E_{2})\
\text{but}\ y\notin \cl_{N}((Y-y)\cap \ell).\]
\end{claim}

\begin{proof}
The minimality of $Y$ means that
\begin{multline*}
r_{1}(Y\cap E_{1})+r_{2}(Y\cap E_{2})-r_{N}(Y\cap \ell)\\
<r_{1}((Y-y)\cap E_{1})+r_{2}((Y-y)\cap E_{2})-r_{N}((Y-y)\cap \ell).
\end{multline*}
and the result follows.
\end{proof}

\begin{claim}
\label{clm:amalgam-clm3}
$|Y\cap \ell| \in \{1,2\}$ and if $|Y\cap\ell|=2$, then $Y\cap \ell$ is independent in $N$.
\end{claim}

\begin{proof}
Assume $y$ and $y'$ are distinct elements of $Y\cap \ell$.
If $\{y, y'\}$ is dependent then we can assume without loss of generality that $y$ is not in $X$, or else $X\cap \ell$ is dependent in $N$.
Since $Y$ contains no loops, it follows that $\{y, y'\}$ is a circuit of $N$.
This means that $y$ is in $\cl_{N}((Y-y)\cap \ell)$, and we have contradiction to Claim~\ref{clm:amalgam-clm1}.
Therefore $\{y,y'\}$ is independent.

We know that $Y-X\subseteq \ell$ is non-empty so $Y\cap\ell$ contains at least one element.
Let $y$ be such an element.
Assume $Y\cap \ell$ contains three distinct elements and 
let $y_{1}$ and $y_{2}$ be distinct elements in $(Y\cap \ell)-y$.
Then $\{y_{1}, y_{2}\}$ is independent by the previous paragraph, so $\{y_{1}, y_{2}\}$ spans $N$.
This means that $y$ is in $\cl_{N}((Y-y)\cap \ell)$, a contradiction.
We conclude that $|Y\cap \ell|<3$.
\end{proof}

Assume $X\cap \ell$ contains distinct elements $x$ and $x'$.
Then $\{x,x'\}$ is independent by Claim~\ref{clm:amalgam-clm3}.
In this case $\{x,x'\}$ spans $N$.
We can let $y$ be an element in $Y-X$, and now $y$ is in $\cl_{N}((Y-y)\cap \ell)$, contradicting Claim~\ref{clm:amalgam-clm1}.
Hence $|X\cap \ell|$ is in $\{0,1\}$.

\begin{claim}
\label{clm:amalgam-clm2}
$X-E_{1}$ is not skew with $\ell$ in $M_{2}$ and $X-E_{2}$ is not skew with $\ell$ in $M_{1}$.
\end{claim}

\begin{proof}
Let $y$ be an arbitrary element of $Y-X$.
Claim~\ref{clm:amalgam-clm1} implies there is a circuit $C$ of $M_{2}$ contained in $Y\cap E_{2}$ that contains $y$.
If $C$ contains an element of $X-E_{1}$, then it certifies that $X-E_{1}$ and $\ell$ are not skew in $M_{2}$.
So assume that $C$ is contained in $Y\cap \ell$.
But now $y$ is contained in $\cl_{N}((Y-y)\cap \ell)$, and we have a contradiction.
The claim follows by symmetry.
\end{proof}

If $\ell$ is contained in $\cl_{1}(X\cap E_{1})$ or $\cl_{2}(X\cap E_{2})$, then Claim~\ref{clm:amalgam-clm2} implies that statement (i) or (ii) holds.
In this case we have nothing left to prove, so we assume that neither $X\cap E_{1}$ nor $X\cap E_{2}$ spans $\ell$.

Assume that $Y\cap \ell$ contains a single element $y$.
Then $(Y-y)\cap E_{1} = X-E_{2}$ and $(Y-y)\cap E_{2} = X-E_{1}$.
Claim~\ref{clm:amalgam-clm1} implies that statement (iii) holds.
Therefore we assume that $|Y\cap \ell|\ne 1$, so Claim~\ref{clm:amalgam-clm3} implies that $Y\cap \ell=\{y,y'\}$ for distinct elements $y$ and $y'$.

Assume that $X\cap \ell$ is non-empty, and therefore contains a single element $x$.
Without loss of generality, we assume that $x=y'$.
Now $X\cap E_{1} = (Y-y)\cap E_{1}$, so $\cl_{1}(X\cap E_{1})$ contains both $x$ and $y$ by Claim~\ref{clm:amalgam-clm1}.
As $\{x,y\}$ is independent in $N$ by Claim~\ref{clm:amalgam-clm3}, it follows that $\cl_{1}(X\cap E_{1})$ contains $\ell$, contradicting our earlier assumption.
Therefore $X\cap \ell$ is empty and $Y\cap E_{1} =(X\cap E_{1})\cup \{y,y'\}$.

We reason as follows.
\begin{align*}
r_{1}(X\cap E_{1}) &< r_{1}((X\cap E_{1}) \cup \ell) \quad
\text{(since $\ell\nsubseteq \cl_{1}(X\cap E_{1})$)}\\
&= r_{1}((X\cap E_{1})\cup \cl_{N}(\{y,y'\}))\\
&\leq r_{1}(\cl_{1}((X\cap E_{1})\cup\{y,y'\}))\\
&= r_{1}((X\cap E_{1})\cup\{y,y'\})\\
&\leq r_{1}((X\cap E_{1})\cup \ell)\\
&< r_{1}(X\cap E_{1}) + 2\quad
\text{(since $X\cap E_{1}$ and $\ell$ are not skew)}
\end{align*}
We see that $r_{1}(X\cap E_{1})+1 = r_{1}((X\cap E_{1})\cup\{y,y'\}) = r_{1}(Y\cap E_{1})$.
Symmetrically, $r_{2}(Y\cap E_{2}) = r_{2}(X\cap E_{2})+1$.
Therefore
\begin{multline*}
|X|>r_{1}(Y\cap E_{1})+r_{2}(Y\cap E_{2})-r_{N}(Y\cap \ell)\\
=r_{1}(X_{1}\cap E_{1})+r_{2}(X\cap E_{2}) = |X_{1}\cap E_{1}|+|X\cap E_{2}| = |X|.
\end{multline*}
This final contradiction completes the proof.
\end{proof}

\begin{definition}
Let $G_{1}$ and $G_{2}$ be graphs such that $V(G_{1})\cap V(G_{2}) = \{u,v\}$.
We use $G_{1}\oplus_{uv} G_{2}$ to denote the graph with vertex-set $V(G_{1})\cup V(G_{2})$ and edge-set $E(G_{1})\cup E(G_{2})$.
Let $e\in E(G_{1})\cup E(G_{2})$ be an edge and let $w\in V(G_{1})\cup V(G_{2})$ be a vertex.
Then $e$ and $w$ are incident in $G_{1}\oplus_{uv} G_{2}$ if and only if they are incident in $G_{1}$ or $G_{2}$.
\end{definition}

We note that $\{u,v\}$ is a vertex cut-set of the graph $G_{1}\oplus_{uv} G_{2}$.

\begin{definition}[Gain-graph amalgam]
\label{def:gain-graph-amalgam}
Let $\Gamma$ be a group, and for $i=1,2$ let $\Omega_{i}=(G_{i},\sigma_{i})$ be a $\Gamma$\dash gain-graph such that $V(G_{1})\cap V(G_{2}) = \{u,v\}$.
Set $\ell$ to be $E(G_{1})\cap E(G_{2})$.
We assume the following conditions hold:
\begin{enumerate}[label = \textup{(\roman*)}]
\item $G_{1}[\ell]=G_{2}[\ell]$,
\item $\sigma_{1}(e,x,y)=\sigma_{2}(e,x,y)$ whenever $(e,x,y)$ is in the domains of both $\sigma_{1}$ and $\sigma_{2}$,
\end{enumerate}
Under these circumstances, the \emph{gain-graph amalgam} $\Omega_{1}\oplus_{uv} \Omega_{2}$ is defined.
Let $\sigma$ be the union of $\sigma_{1}$ and $\sigma_{2}$.
Then $\Omega_{1}\oplus_{uv} \Omega_{2}$ is the $\Gamma$\dash gain-graph $(G_{1}\oplus_{uv} G_{2}, \sigma)$.
We say that $\{u,v\}$ is the \emph{base} of the amalgam.
\end{definition}

\begin{lemma}\label{lem:graph-amalgam}
Let $\Gamma$ be a group and for $i=1,2$, let $\Omega_{i}=(G_{i},\sigma_{i})$ be a $\Gamma$\dash gain-graph.
Assume that $V(G_{1})\cap V(G_{2}) = \{u,v\}$ and let $\ell= E(G_{1})\cap E(G_{2})$.
We assume the following conditions hold.
\begin{enumerate}[label = \textup{(\roman*)}]
\item $G_{1}[\ell]=G_{2}[\ell]$,
\item $\sigma_{1}(e,x,y)=\sigma_{2}(e,x,y)$ whenever $(e,x,y)$ is in the domains of both $\sigma_{1}$ and $\sigma_{2}$,
\item $\ell$ contains unbalanced loop edges incident with $u$ and $v$, and
\item whenever $P_{i}$ is a path of $G_{i}$ from $u$ to $v$ for each $i=1,2$, and $\sigma_{1}(P_{1})=\sigma_{2}(P_{2})$, then there is a non-loop edge $e\in \ell$ such that $\sigma_{1}(e,u,v)=\sigma_{2}(e,u,v)=\sigma_{1}(P_{1})=\sigma_{2}(P_{2})$.
\end{enumerate}
Then $F(\Omega_{1}\oplus_{uv} \Omega_{2})=\amal{F(\Omega_{1})}{F(\Omega_{2})}[\ell]$.
\end{lemma}

\begin{proof}
Let $G$ be $G_{1}\oplus_{uv} G_{2}$ and let $\sigma$ be the union of $\sigma_{1}$ and $\sigma_{2}$.
We will prove the lemma by showing that a circuit of $F(G,\sigma)$ is dependent in \amal{F(\Omega_{1})}{F(\Omega_{2})} and a circuit in \amal{F(\Omega_{1})}{F(\Omega_{2})} is dependent in $F(G,\sigma)$.
For $i=1,2$, let $E_{i}$ stand for $E(G_{i})$.

We start by assuming that $C$ is a circuit of $F(G,\sigma)$.
This means that $C$ is either a balanced cycle in $(G,\mathcal{B}(\sigma))$, or it is a bicycle containing no balanced cycles.
If $C$ is contained in $G_{i}$ for some $i$, then $C$ is dependent in $F(\Omega_{i})$.
In this case $C$ is dependent in \amal{F(\Omega_{1})}{F(\Omega_{2})}, by Proposition~\ref{prop:amalgams}.
So we will assume that both $C-E_{1}$ and $C-E_{2}$ are non-empty.

First assume that $C$ is a theta subgraph consisting of three paths between vertices $x$ and $y$.
Any cut-set in this subgraph consisting of two vertices is contained in one of the three paths from $x$ to $y$.
Therefore $u$ and $v$ are both on the same path of $C$ from $x$ to $y$.
From this it follows that we can assume without loss of generality that $C-E_{1}$ is a path joining $u$ to $v$.
Therefore $C\cap E_{1}$ contains an unbalanced cycle made of two of the paths in $C$ joining $x$ to $y$.
Let $q$ be an unbalanced loop incident with either $u$ or $v$.
Now $(C\cap E_{1})\cup q$ is connected and contains two unbalanced cycles, including $q$.
This shows that $q$ is in $\cl_{1}(C\cap E_{1})$.
Since we have a guarantee that $\ell$ contains unbalanced loops incident with $u$ and $v$, it now follows that $\cl_{1}(C\cap E_{1})$ contains $\ell$.
The union of $C-E_{1}$ with loops incident with $u$ and $v$ forms a circuit contained in $\ell\cup(C-E_{1})$.
This circuit contains elements from both $\ell$ and $C-E_{1}$.
Therefore these sets are not skew in $F(\Omega_{2})$.
Now Proposition~\ref{prop:amalgams} tells us that $C$ is dependent in \amal{F(\Omega_{1})}{F(\Omega_{2})}.

Next assume that $C$ is a handcuff.
Then $C$ contains two edge-disjoint cycles, both unbalanced.
Let these cycles be $C_{1}$ and $C_{2}$ and let $P$ be the path of $C$ that joins a vertex of $C_{1}$ to a vertex of $C_{2}$.
Note that $P$ may comprise a single vertex.
Assume that both $C-E_{1}$ and $C-E_{2}$ contain cycles of $C$.
Then $P$ contains a vertex $w\in\{u,v\}$.
Let $q\in \ell$ be an unbalanced loop that is incident with $w$.
Then $C-E_{1}$ contains an unbalanced cycle and a path joining this cycle to $w$, so the union of $q$ and $C- E_{1}$ contains a handcuff with two unbalanced cycles.
This shows that $q$ is in $\cl_{2}(C-E_{1})$.
The same argument shows that $q$ is also in $\cl_{1}(C-E_{2})$.
Proposition~\ref{prop:amalgams} now shows that $C$ is dependent in \amal{F(\Omega_{1})}{F(\Omega_{2})}.
Therefore we assume without loss of generality that $C-E_{1}$ does not contain a cycle.

Now $G[C-E_{1}]$ is a forest with at least one edge, and therefore it contains at least two degree-one vertices.
But $G[C]$ contains no such vertex, so any degree-one vertex in $G[C-E_{1}]$ is incident with edges in both $C-E_{1}$ and $C\cap E_{1}$.
There are at most two vertices ($u$ and $v$) incident with edges in both these sets.
This shows that $G[C-E_{1}]$ contains exactly two degree-one vertices, and these vertices are $u$ and $v$.
Thus $C-E_{1}$ is a subpath of $G[C]$ and its end vertices are $u$ and $v$.
No internal vertex of this path has degree three in $G[C]$, or else it would be a vertex incident with edges in both $C-E_{1}$ and $C\cap E_{1}$, and the only such vertices are $u$ and $v$.
In particular, the vertex in both $C_{1}$ and $P$ is not an internal vertex of $C-E_{1}$.
The same applies to the vertex in both $C_{2}$ and $P$.
Now, up to symmetry, there are two possibilities: the path $C-E_{1}$ is contained in $C_{1}$, or is contained in $P$.

First consider the case that $C-E_{1}$ is contained in $C_{1}$.
Then $G[C\cap E_{1}]$ contains the unbalanced cycle $C_{2}$ as well as the vertices $u$ and $v$.
Since $\ell$ contains unbalanced loops incident with $u$ and $v$, it follows that $\cl_{1}(C\cap E_{1})$ contains $\ell$.
Also, $C-E_{1}$ is a path joining $u$ and $v$, and therefore we can find a circuit of $F(\Omega_{2})$ contained in $(C-E_{1})\cup\ell$ that contains elements from both $C-E_{1}$ and $\ell$.
Again we see that $C$ is dependent in \amal{F(\Omega_{1})}{F(\Omega_{2})}.

Next we assume that $C-E_{1}$ is a subpath of $P$.
One component of $C\cap E_{1}$ contains an unbalanced cycle and $u$.
The other contains an unbalanced cycle and $v$.
As before, we can argue that $\cl_{1}(C\cap E_{1})$ contains $\ell$, and that $\ell$ and $C-E_{1}$ are not skew, so once again $C$ is dependent in \amal{F(\Omega_{1})}{F(\Omega_{2})}.

The remaining case is that $C$ is a balanced cycle of $(G,\sigma)$.
In this case $C$ contains both $u$ and $v$.
For $i=1,2$, let $P_{i}$ be the path of $G_{i}$ from $u$ to $v$ that is contained in $C$.
Since $C$ is balanced, we see that $\sigma(C)$ is the identity, and therefore $\sigma_{1}(P_{1})=\sigma_{2}(P_{2})$.
The hypotheses mean that there is an edge $e$ in $\ell$ joining $u$ to $v$ such that $\sigma(e,u,v)=\sigma_{1}(P_{1})=\sigma_{2}(P_{2})$.
This means that the union of $P_{i}$ with $e$ is a balanced cycle of $\Omega_{i}$.
So $e$ is in $\cl_{i}(C-E_{3-i})$.
Proposition~\ref{prop:amalgams} now tells us that $C$ is dependent in \amal{F(\Omega_{1})}{F(\Omega_{2})}.

We have concluded the argument that a circuit of $F(G,\sigma)$ is dependent in \amal{F(\Omega_{1})}{F(\Omega_{2})}.
Now we will assume that $C$ is a circuit of \amal{F(\Omega_{1})}{F(\Omega_{2})}.
If either $C-E_{1}$ or $C-E_{2}$ is empty, then $C$ is a dependent subset of $F(\Omega_{1})$ or $F(\Omega_{2})$.
In this case $C$ is dependent in $F(G,\sigma)$ and we are done.
So we will assume that $C-E_{1}$ and $C-E_{2}$ are both non-empty.
Therefore $C\cap E_{i}$ is an independent subset of $F(\Omega_{i})$ for $i=1,2$.
Now we can apply Proposition~\ref{prop:amalgams} and deduce that statement (i), (ii), or (iii) from that result holds.

Symmetrical arguments will deal with both (i) and (ii), so we assume the former holds.
Then $\ell$ is contained in $\cl_{1}(C\cap E_{1})$ and $C-E_{1}$ is not skew with $\ell$ in $F(\Omega_{2})$.
Let $w$ be an arbitrary vertex in $\{u,v\}$ and let $q$ be an unbalanced loop incident with $w$.
If $q$ is in $C\cap E_{1}$, then there is a connected component of $C\cap E_{1}$ that contains an unbalanced cycle (namely $q$) and $w$.
If $q$ is not in $C\cap E_{1}$, then $(C\cap E_{1})\cup q$ contains a circuit that contains $q$, and this circuit must be a handcuff.
In either case $C\cap E_{1}$ contains a connected component that contains both an unbalanced cycle and $w$.
We choose such a component and call it $H_{w}$.

Now let $C'$ be a circuit of $F(\Omega_{2})$ that is contained in $\ell\cup (C-E_{1})$ and which contains edges from both $\ell$ and $C-E_{1}$.
Since $G[C']$ is connected, it follows that $G[C'-E_{1}]$ contains at least one of $u$ and $v$.
Assume that $G[C'-E_{1}]$ contains a cycle and the vertex $u$.
Then the union of $C'-E_{1}$ and $H_{u}$ contains a connected component that contains two distinct cycles.
Now it follows that $C$ contains a circuit of $F(G,\sigma)$.
The same argument applies if $C'-E_{1}$ contains $v$ instead of $u$.
Therefore we must assume that $G[C'-E_{1}]$ is a forest.
Then $C'-E_{1}$ contains at least two degree-one vertices.
Since $C'$ contains no such vertex, it follows that any degree-one vertex of $C'-E_{1}$ must share a common vertex with an edge in $\ell$.
Therefore $C'-E_{1}$ has exactly two degree-one vertices, and in fact it is a path of $G_{2}$ between $u$ and $v$.
We consider the union of this path with $H_{u}$ and $H_{v}$.
Note that the union is a connected subgraph of $C$.
If $H_{u}$ and $H_{v}$ are distinct, then this connected subgraph contains two distinct cycles, and therefore $C$ is dependent in $F(G,\sigma)$.
So we assume $H_{u}$ and $H_{v}$ are the same connected component.
This component contains a path of $G_{1}$ from $u$ to $v$.
The union of this path with $C'-E_{1}$ is a cycle and this cycle is distinct from the cycle of $C\cap E_{1}$ contained in $H_{u}=H_{v}$.
We have once again found a connected component of $G[C]$ that contains two distinct cycles, so $C$ is dependent in $F(G,\sigma)$.
We have now concluded the argument when case (i) holds in Proposition~\ref{prop:amalgams}.

We can now assume that case (iii) holds.
This means we can choose an edge $e\in \ell$ and for each $i=1,2$, we can let $C_{i}$ be a circuit of $F(\Omega_{i})$ such that $e\in C_{i}\subseteq (C-E_{3-i})\cup e$.
Assume that $e$ is a loop edge incident with $w\in \{u,v\}$.
Then each $C_{i}$ is a handcuff, and $C- E_{3-i}$ contains a connected component that contains a cycle and the vertex $w$.
Now $G[C]$ contains a connected component with two distinct cycles and we are done.
So we assume $e$ is an edge joining $u$ and $v$.
Then $C-E_{3-i}$ contains a path from $u$ to $v$.
This means that there is a cycle of $C$ containing edges from both $C-E_{1}$ and $C-E_{2}$.
If either $C_{1}-e$ or $C_{2}-e$ contains a cycle, then $C$ contains a component that contains two cycles, so we assume that both $C_{1}-e$ and $C_{2}-e$ are forests.
Since $C_{1}$ and $C_{2}$ have no degree-one vertices, it now follows that both these subgraphs are balanced cycles.
For $i=1,2$, let $P_{i}$ be the path $C_{i}-e$, directed from $u$ to $v$.
Because $C_{1}$ and $C_{2}$ are balanced, we conclude that $\sigma_{1}(P_{1})=\sigma(e,u,v)=\sigma_{2}(P_{2})$.
This means that the union of $P_{1}$ and $P_{2}$ is a balanced cycle of $(G,\sigma)$.
Therefore $C$ contains a circuit of $F(G,\sigma)$ and this completes the proof.
\end{proof}

\begin{lemma}\label{lem:amalgam-is-biased}
For $i=1,2$, let $\Omega_{i}=(G_{i},\mathcal{B}_{i})$ be a biased graph without balanced loops and assume that $V(G_{1})\cap V(G_{2})=\{u,v\}$.
Assume also that in $\Omega_{1}$ and $\Omega_{2}$, every vertex is incident with an unbalanced loop, and that each vertex $w$ has two distinct neighbours $w_{1}$ and $w_{2}$ such that for each $i=1,2$, there is an unbalanced $2$\dash edge cycle containing $w$ and $w_{i}$.
Let $\Omega=(G,\mathcal{B})$ be a biased graph such that $G$ has no isolated vertices and $\amal{F(\Omega_{1})}{F(\Omega_{2})}$ is equal to $F(\Omega)$.
Then $G$ is isomorphic to $G_{1}\oplus_{uv}G_{2}$.
\end{lemma}

\begin{proof}
Let $L$ be a set containing exactly one unbalanced loop incident with each vertex in $G_{1}\oplus_{uv}G_{2}$.
Then $L\cap E_{i}$ is a basis of $F(\Omega_{i})$ for each $i=1,2$.
It is easy to verify, using Proposition~\ref{prop:amalgams}, that $L$ is independent in $\amal{F(\Omega_{1})}{F(\Omega_{2})}$.
Let $e$ be an arbitrary element of $\amal{F(\Omega_{1})}{F(\Omega_{2})}$ that is not in $L$.
We can assume without loss of generality that $e$ is an edge of $\Omega_{1}$.
Either $e$ is an unbalanced loop that has a common vertex with a loop in $L$, or it is non-loop edge of $\Omega_{1}$.
In these cases, $e$ is in a $2$\dash\ or $3$\dash element circuit of $F(\Omega_{1})$ that is contained in $L\cup e$.
This circuit is also a circuit of $\amal{F(\Omega_{1})}{F(\Omega_{2})}$.
It now follows that $L$ spans $\amal{F(\Omega_{1})}{F(\Omega_{2})}$, and is therefore a basis of this matroid.
Let $x$ and $y$ be an arbitrary pair of distinct vertices contained in an unbalanced $2$\dash edge cycle of either $\Omega_{1}$ or $\Omega_{2}$.
The edges from the $2$\dash edge cycle along with unbalanced loops incident with $x$ and $y$ form a long line of either $F(\Omega_{1})$ or $F(\Omega_{2})$, and it is not difficult to see that this line is also a line in $\amal{F(\Omega_{1})}{F(\Omega_{2})}$.
Now the hypotheses imply that each element of $L$ is in two distinct long lines of $\amal{F(\Omega_{1})}{F(\Omega_{2})}$.
Now Proposition~\ref{prop:long-lines} implies that the elements of $L$ are all unbalanced loops in $\Omega$.
Since $L$ is independent in $F(\Omega)$, we have no more than one element of $L$ incident with any given vertex of $G$, for any two loops incident with the same vertex form a handcuff.
If there exists some vertex of $G$ not incident with a loop in $L$, then this vertex is incident with an edge $e$ (since $G$ has no isolated vertices), and now $e$ is not spanned by $L$.
This is a contradiction, so every vertex of $G$ is incident with exactly one loop in $L$.
This induces a bijection between the vertices of $G_{1}\oplus_{uv}G_{2}$ and $G$.
Any edge of $G_{1}\oplus_{uv} G_{2}$ that is not in $L$ is incident with at most two elements of $L$, and this shows that the bijection between the vertices of $G_{1}\oplus_{uv}G_{2}$ and $G$ gives us an isomorphism between the two graphs.
\end{proof}

\section{Uniformly locally finite groups}
\label{sec:ULF}

In this section, we will prove that if $\Gamma$ is not a uniformly locally finite group, the class of $\Gamma$\dash gain-graphic matroids is not \cmso\dash definable.
\begin{definition}
Let $\Gamma$ be a group.
A subset $\{a_1, \ldots, a_n\}$ is a \emph{generating set} if it is closed under inverses. Then, for any element $g \in \langle a_1, \ldots, a_n \rangle$ there is a string of the $\{a_i\}$ which evaluates to $g$. Let $f_{\{a_i\}}(g)$ be the minimum length of such a string (letting $f_{\{a_i\}}(\id) = 0$). Note that for $g_1, g_2 \in \langle a_1, \ldots, a_n \rangle$, $f_{\{a_i\}}(g_1g_2) \leq f_{\{a_i\}}(g_1)+f_{\{a_i\}}(g_2)$.
When there is only one generating set in the context, we will drop the subscript and write $f(x) = f_{\{a_i\}}(x)$
\end{definition}

\begin{definition}\label{def:ulf-gadget}
For $n,N \in \bN$, the graph $H_{n,N}^*$ is defined in Figure~\ref{fig:HnN*}.
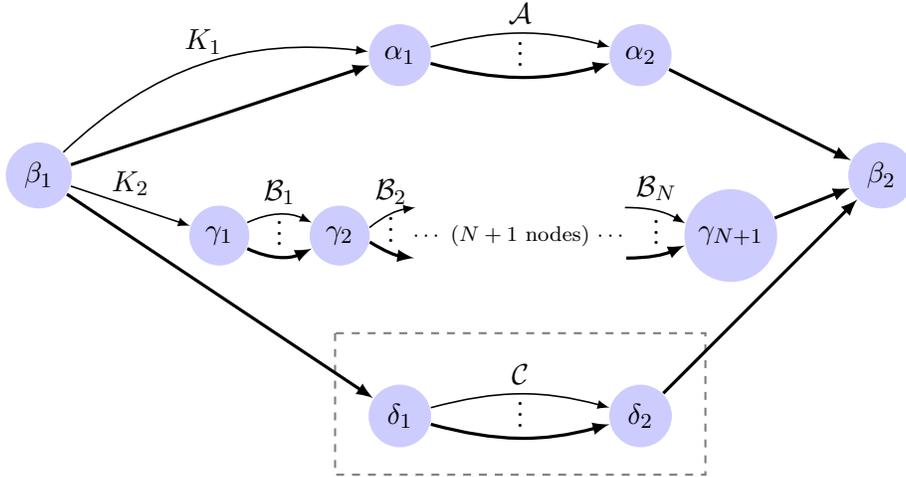
\begin{figure}[htb]
\centering
\begin {tikzpicture}[-latex ,auto,on grid , scale=0.8,
semithick ,
state/.style ={ circle ,top color =white , bottom color = processblue!20 ,
draw,processblue , text=blue , minimum width =1 cm}]
  \node [style1] (n11) at (18,0) {$\alpha_1$};
  \node [style1] (n12) at (22,0) {$\alpha_2$};
  \node [style1] (n21) at (12,-2)  {$\beta_1$};
  \node [style1] (n22) at (26,-2)  {$\beta_2$};
  \node [style1] (n31) at (15,-3)  {$\gamma_1$};
  \node [style1] (n32) at (17,-3)  {$\gamma_2$};
  \node [style3] (n33) at (20,-3)  {\scriptsize $\cdots$ ($N+1$ nodes) $\cdots$};
  \node [style1] (n35) at (23.5,-3)  {$\gamma_{N+1}$};
  \node [style1] (n41) at (18,-6) {$\delta_1$};
  \node [style1] (n42) at (22,-6)  {$\delta_2$};

\draw[gray,thick,dashed] ($(n41.north west)+(-0.7,1.0)$)  rectangle ($(n42.south east)+(0.7,-0.6)$);

\path (n11) edge [bend left =15] node[above] {$\mathcal{A}$} (n12);
\path (n11) edge [bend left =-15,very thick] node[above] {$\vdots$} (n12);
\path (n21) edge [bend left =25] node[above] {$K_1$} (n11);
\path (n21) edge [bend left =0,very thick] node[below] {} (n11);
\path (n12) edge [bend left =0,very thick] node[below] {} (n22);
\path (n21) edge [bend left =0] node[above] {$K_2$} (n31);
\path (n35) edge [bend left =0,very thick] node[above] {} (n22);
\path (n21) edge [bend left =0,very thick] node[above] {} (n41);
\path (n42) edge [bend left =0,very thick] node[above] {} (n22);

\path (n41) edge [bend left =-15,very thick] node[above] {$\vdots$} (n42);
\path (n41) edge [bend left =15] node[above] {$\mathcal{C}$} (n42);

\path (n31) edge [bend left =-25,very thick] node[above] {$\vdots$} (n32);
\path (n31) edge [bend left =25] node[above] {$\mathcal{B}_1$} (n32);

\path (n32) edge [bend left =-12,very thick] node[above] {$\vdots$} (n33);
\path (n32) edge [bend left =15] node[above] {$\mathcal{B}_2$} (n33);

\path (n33) edge [bend left =-12,very thick] node[above] {$\vdots$} (n35);
\path (n33) edge [bend left =15] node[above] {$\mathcal{B}_{N}$} (n35);

\end{tikzpicture}
\caption{The graph $H_{n,N}^{*}$.}
\label{fig:HnN*}
\end{figure}

For ease of notation, we will refer to vertices by lower-case Greek letters, and edges by upper-case Roman letters. Let $G_{\ell,N}^*$ be the subgraph enclosed by the dotted line, and let $\ell_N^*$ be its edge-set. Let $\mathcal{T}$ be the collection of bolded edges. Note that each of $\mathcal{A},\mathcal{B}_1, \ldots, \mathcal{B}_N$ and $\mathcal{C}$ are collections of edges:
\begin{align*}
&\mathcal{A} = \{A_\id, A_1, \ldots, A_n,A_s\}\\
&\mathcal{B}_i = \{B_{i,\id},B_{i,1}, \ldots, B_{i,n}\} \text{ for all }1\leq i\leq N\\
&\mathcal{C} = \{C_{\id}, C_1, \ldots, C_n,C_s\}
\end{align*}
$\mathcal{T}$ has a single element in each of these: the elements $A_\id,B_{i,\id} : 1 \leq i\leq N$, and $C_\id$. Therefore, we have added a single bold line in each collection in the diagram.

Next, we define $H_{n,N}$ as the union of $H^*_{n,N}$ and the following new edges:

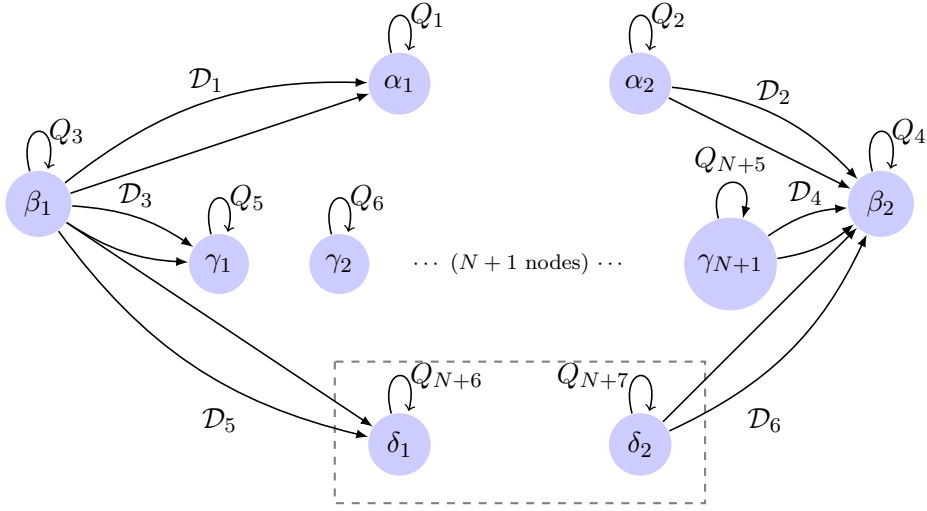
\begin{figure}[htb]
\centering
\begin {tikzpicture}[-latex ,auto,on grid , scale=0.8,
semithick ,
state/.style ={ circle ,top color =white , bottom color = processblue!20 ,
draw,processblue , text=blue , minimum width =1 cm}]
  \node [style1] (n11) at (18,0) {$\alpha_1$};
  \node [style1] (n12) at (22,0) {$\alpha_2$};
  \node [style1] (n21) at (12,-2)  {$\beta_1$};
  \node [style1] (n22) at (26,-2)  {$\beta_2$};
  \node [style1] (n31) at (15,-3)  {$\gamma_1$};
  \node [style1] (n32) at (17,-3)  {$\gamma_2$};
  \node [style3] (n33) at (20,-3)  {\scriptsize $\cdots$ ($N+1$ nodes) $\cdots$};
  \node [style1] (n35) at (23.5,-3)  {$\gamma_{N+1}$};
  \node [style1] (n41) at (18,-6) {$\delta_1$};
  \node [style1] (n42) at (22,-6)  {$\delta_2$};

\draw[gray,thick,dashed] ($(n41.north west)+(-0.7,1.0)$)  rectangle ($(n42.south east)+(0.7,-0.6)$);

\path (n21) edge [bend left =20] node[above] {$\mathcal{D}_1$} (n11);
\path (n21) edge [bend left =0] node[above] {} (n11);
\path (n12) edge [bend left =20] node[above] {$\mathcal{D}_2$} (n22);
\path (n12) edge [bend left =0] node[above] {} (n22);
\path (n21) edge [bend left =15] node[above] {$\mathcal{D}_3$} (n31);
\path (n21) edge [bend left =-15] node[above] {} (n31);
\path (n35) edge [bend left =15] node[above] {$\mathcal{D}_4$} (n22);
\path (n35) edge [bend left =-15] node[above] {} (n22);
\path (n21) edge [bend left =-20] node[below] {} (n41);
\path (n21) edge [bend left =0] node[below] {\begin{tabular}{c}
    ${}$  \\
    ${}$  \\
    $\mathcal{D}_5$
\end{tabular}} (n41);
\path (n42) edge [bend left =-20] node[below] {} (n22);
\path (n42) edge [bend left =0] node[below] {\begin{tabular}{c}
    ${}$  \\
    ${}$  \\
    \ $\mathcal{D}_6$
\end{tabular}} (n22);

\path (n11) edge [loop above] node[right] {$Q_1$} (n11);
\path (n12) edge [loop above] node[right] {$Q_2$} (n12);
\path (n21) edge [loop above] node[right] {$Q_3$} (n21);
\path (n22) edge [loop above] node[right] {$Q_4$} (n22);
\path (n31) edge [loop above] node[right] {$Q_5$} (n31);
\path (n32) edge [loop above] node[right] {$Q_6$} (n32);
\path (n35) edge [loop above,style={min distance=8mm}] node[above] {$Q_{N+5}$} (n35);
\path (n41) edge [loop above] node[right] {$Q_{N+6}$} (n41);
\path (n42) edge [loop above] node[left] {$Q_{N+7}$} (n42);

\end{tikzpicture}
\caption{The graph $H_{n,N} \setminus H_{n,N}^{*}$.}
\label{fig:HnN}
\end{figure}

Each $\mathcal{D}_i : 1 \leq i \leq 6$ is a collection of two edges: $\mathcal{D}_i = \{D_{i,1}, D_{i,2}\}$. Each $Q_i$ is a single loop edge. Let $G_{\ell,N} = G_{\ell,N}^* \cup \{Q_{N+6},Q_{N+7}\}$, and let $\ell_N$ be its edge-set. We will use these new edges at only one point in the argument.

Suppose we have some group $\Gamma$, a generating set $\{a_1, \ldots, a_n\}$, some $s, \mathfrak{M} \in\langle a_1, \ldots, a_n \rangle$. Then, $H^*_{n,N}$ has a $\Gamma$\dash gaining $\sigma^* = \sigma^*(\Gamma,\{a_1, \ldots, a_n\},s,\mathfrak{M})$, defined as follows:
\begin{align*}
    &\sigma^*(T) = \id \ \forall \ T \in \mathcal{T}\\
    &\sigma^*(A_j) = a_j \ \forall \ 1 \leq j \leq n\\
    &\sigma^*(B_{i,j}) = a_j \ \forall \ 1 \leq j \leq n, 1 \leq i \leq N\\
    &\sigma^*(C_j) = a_j \ \forall \ 1 \leq j \leq n
\end{align*}
And $\sigma^*(K_1) = \sigma^*(K_2)=\mathfrak{M}$, $\sigma^*(A_s)=\sigma^*(C_s)=s$. Note all these edges are oriented as in the diagrams, so that, for example, $\sigma^{*}(A_{j})=a_{j}$ is shorthand for $\sigma^{*}(A_{j},\alpha_{1},\alpha_{2})=a_{j}$.
Secondly, given elements
\[d_{i,j} : 1 \leq i \leq 6,\ 1 \leq j \leq 2,
\quad \text{and} \quad q_i : 1 \leq i \leq N+7\]
of $\Gamma$, we can define $\sigma = \sigma(\Gamma,\{a_1, \ldots, a_n\},s,\mathfrak{M}, \{d_{i,j}\}, \{q_i\})$, a $\Gamma$\dash gaining of $H_{n,N}$ extending $\sigma^*$, by
\begin{align*}
    &\sigma(D_{i,j}) = d_{i,j} \ \forall \ 1 \leq i \leq 6,\ 1 \leq j \leq 2\\
    &\sigma(Q_i) = q_i \ \forall \ 1 \leq i \leq N+7
\end{align*}
These edges are also oriented as in the diagrams. We will always assume
\begin{equation}
\text{\begin{tabular}{c}
    $f_{\{a_i\}}(s)=N$,\\
    $f_{\{a_i\}}(\mathfrak{M}) \geq 2N+1$, and\\
     the $\{d_{i,j}\}$, and $\{q_i\}$ are chosen so that \\
     no cycle containing a $D_{i,j}$ or $Q_i$ is balanced
\end{tabular}} \tag{$\dagger$}
\end{equation}
Subject to this assumption, the balanced cycles of $\sigma$ depend only on the choice of $\{a_1, \ldots, a_n\}$ and $s$.

Let us also define $\Omega(\Gamma,\{a_1, \ldots, a_n\},s,\mathfrak{M}, \{d_{i,j}\}, \{q_i\})$ as the induced biased graph on $H_{n,N}$.
\end{definition}

\begin{remark}\label{rmk:can-find-monsters}
    Let $\Gamma$ be an infinite group, and let $S$ be a finite collection of strings over the characters $\Gamma \cup \{x\}$, where each string contains exactly one copy of $x$. Given an element $g \in \Gamma$ and $s \in S$, there is an evaluation $s(g) \in \Gamma$, where we replace $x$ with $g$. Since $\Gamma$ is infinite, there will always be an element $g \in \Gamma$ such that for all $s \in S$, $s(g) \neq \id$.
    Thus, we can always find elements satisfying condition $(\dagger)$ in Definition~\ref{def:ulf-gadget}.
\end{remark}
Since we will use it several times, we note a specialisation of [13, Lemma 5.3]:
\begin{fact}\label{fact:identity-tree}
Let $\Gamma$ be a group, $G$ a finite graph, and $\sigma$ a $\Gamma$\dash gaining of $G$. Let $G_2$ be a subgraph such that in $(G_2,\sigma|_{G_2})$ every cycle is balanced. Then, there is a $\Gamma$\dash gaining $\sigma_2$ of $G$, with the same balanced and unbalanced cycles as $\sigma$, such that for all $e \in G_2$, $\sigma_2(e) = \id$.
\end{fact}

\begin{lemma}\label{lem:Myhill-Nerode-application}
    Let $\mathcal{M}$ be a \cmso-definable class of matroids. Let $\Gamma$ be a group, and let $\{(\Xi_j,\sigma_j) : j \in J\}$ be a collection of $\Gamma$\dash gained finite graphs, over an index set $J$, all containing a fixed copy of a fixed graph $G_\ell$, with edge-set $\ell$. Thus, without loss of generality, we may take $G_\ell = \Xi_i \cap \Xi_j$ for any $i \neq j$.
Assume also that $G_\ell$ has exactly two vertices, and in every $(\Xi_j,\sigma_j)$ each vertex has an unbalanced loop edge. Then, there is a finite partition $J = J_1 \cup \cdots\cup J_n$ such that for all $l$, for all $j_1,j_2 \in J_l$, and for all $k$, 
    \[\amal{F(\Xi_{j_1},\sigma_{j_1})}{F(\Xi_k,\sigma_k)}[\ell] \in \mathcal{M} \leftrightarrow \amal{F(\Xi_{j_2},\sigma_{j_2})}{F(\Xi_k,\sigma_k)}[\ell] \in \mathcal{M}.\]
    That is, whether the amalgam is in the class depends only on where in the partition the amalgam components reside.
\end{lemma}
\begin{proof}
    Let $\cl_j$ be the matroid closure operator of $F(\Xi_j,\sigma_j)$, and let $E_j$ be the edge-set of $\Xi_j$ for each $j$. We finitely colour each $\mathcal{P}(E_j)$ as follows:    Given $X \subseteq E_j$,
    \begin{multline*}
        t_j(X) = (\cl_j(X) \cap \ell, \cl_j(X \setminus \ell) \cap \ell,\\ \text{Is $X$ dependent?}, \text{Is $X \setminus \ell$ skew with $\ell$?})     
    \end{multline*}
    Where the last two terms are boolean values. Let $T$ be the set of colours (which does not depend on $j$). The pairs $(E_j, t_j)$ are then $T$-coloured systems.

    Note that for any $j,k$, and any set $X \cup Y \subseteq \amal{F(\Xi_{j},\sigma_{j})}{F(\Xi_k,\sigma_k)}[\ell]$, where $X \subseteq E_j$ and $Y \subseteq E_k \setminus \ell$, whether $X \cup Y$ is dependent depends only on $t_j(X)$ and $Y$, by Proposition~\ref{prop:amalgams}. Thus, for each $j$, let
    \[d_j : \mathcal{P}(E_j \setminus \ell) \times T \rightarrow \{0,1\}\]
    be the function which, given a subset $Y$ and colour $c \in T$, returns whether $X \cup Y$ would be independent for any $X \subseteq E_k$ such that $t_k(X) = c$. Given this definition, for any $j,k$, we have
    \[\amal{F(\Xi_{j},\sigma_{j})}{F(\Xi_k,\sigma_k)}[\ell] \cong (E_j,t_j)\boxplus(E_k,d_k).\]
    By Lemma~\ref{thm:sharp-Myhill-Nerode}, there is a finite partition $J = J_1 \cup \cdots\cup J_n$ such that for all $l \in [n]$, all $j_1,j_2,k \in J$, \[(E_{j_1},t_{j_1})\boxplus(E_k,d_k) \in \mathcal{M} \leftrightarrow (E_{j_2},t_{j_2})\boxplus(E_k,d_k) \in \mathcal{M}.\]
    By the above isomorphism, this concludes the proof.
\end{proof}

The next theorem shows that an infinite group $\Gamma$ must be locally finite in order for the class of $\Gamma$\dash gain-graphic matroids to be \cmso\dash definable.

\begin{theorem}\label{thm:uniformly-locally-finite-2}
Let $\Gamma$ be a group, and let $\{a_1, \ldots, a_n\} \subseteq \Gamma$ be a generating set which generates an infinite subgroup. Then the class of $\Gamma$\dash gain-graphic matroids is not \cmso\dash definable.
\end{theorem}

\begin{proof}
We may restrict to the case where the $\{a_i\}$ contain two elements $\alpha_1 \neq\alpha_2$ such that $\alpha_1^{-1}\neq \alpha_2$, since we may freely add new pairs $\{g,g^{-1}\} \subseteq \langle a_1 \ldots a_n \rangle$ to the generating set.
For each $1 \leq N < \omega$, fix some
\[s_N,\mathfrak{M}_N, d_{i,j} : 1 \leq i \leq 6,\ 1 \leq j \leq 2\quad\text{and}\quad q_i : 1 \leq i \leq 10\]
all in $\Gamma$ satisfying condition $(\dagger)$ in Definition~\ref{def:ulf-gadget}.
Let $\sigma_N$ be $\sigma(\Gamma,\{a_1, \ldots, a_n\},s_N,\mathfrak{M}_N,\{d_{i,j}\},\{q_i\})$, the $\Gamma$-gaining of $H_{n,N}$.
Then, for ease of notation, for each $N \geq 1$, let $H_N = (H_{n,N},\sigma_N)$, and let $\Omega_N$ be the associated biased graph. Note that any $H_N$ contains $G_{\ell,N}$, and the $\{G_{\ell,N}\}$ are all isomorphic as $\Gamma$\dash gained graphs. So, we could arrange that the $H_n$ are all sub-graphs of some larger graph $\mathbb{G}$, and that the $G_{\ell,N}$ are actually equal, and for any $N,N'$ distinct, $H_N \cap H_{N'} = G_{\ell,N}$. Thus, given any two $H_N,H_{N'}$, we can construct both the matroid amalgam $\amal{F(H_N)}{F(H_{N'})}$ over the base $F(\ell_N) = F(\ell_{N'})$, and the graph amalgam $H_{n,N}\oplus_{\delta_1,\delta_2} H_{n,N'}$ over the base $\{\delta_1,\delta_2\} = \{\delta_1',\delta_2'\}$.
Throughout the proof, when we amalgamate it will always be over these bases.
By Lemma~\ref{lem:Myhill-Nerode-application}, if the class of $\Gamma$\dash gain-graphic matroids is \cmso\dash definable, there is a finite partition of $\omega$ such that for any $N,N'$, whether $\amal{F(H_N)}{F(H_{N'})}$ is a $\Gamma$\dash gain-graphic matroid depends only on which classes in the partition $N$ and $N'$ belong to. Thus, for a contradiction it suffices to show that the amalgam $\amal{F(H_N)}{F(H_{N'})}$ is a $\Gamma$\dash gain matroid if and only if $N=N'$ for all $N,N' \geq 1$. There are two directions to this.

First, suppose $\amal{F(H_N)}{F(H_{N'})}$ is a $\Gamma$\dash gain-graphic matroid, built from some $\Gamma$-gained graph $(X,\tau_0)$, with associated biased graph $\Omega$. Without loss of generality, we may assume $X$ has no isolated vertices. Note that $X$ has the same edge-set as $Y = H_{n,N}\oplus H_{n,N'}$. We claim that the triple $\Omega_N,\Omega_N',\Omega$ satisfies the conditions of Lemma~\ref{lem:amalgam-is-biased}.
Certainly no loop edges in $H_{N}$ or $H_{N'}$ are balanced.
The desired unbalanced loops are supplied by $Q_1, \ldots, Q_{N+7},Q_1', \ldots, Q_{N'+7}'$.
Next, for the desired unbalanced cycles, note that by assumption each $\mathcal{A},\mathcal{A}',\mathcal{B}_i,\mathcal{B}_i',\mathcal{C},\mathcal{C}'$ contains an unbalanced cycle, since we assumed $\{a_i\}$ contains two distinct elements which are not reciprocals of each other. Note also by property $(\dagger)$ in Definition~\ref{def:ulf-gadget}, each $\mathcal{D}_i$ and $\mathcal{D}_i'$ contains an unbalanced cycle.
Finally, by assumption $\amal{F(H_N)}{F(H_{N'})} = F(X,\tau_0)$ and $X$ contains no isolated vertices. Thus, applying the lemma, we conclude $X \cong Y$.
So, by our assumption, there is a $\Gamma$\dash gaining $\tau_0$ of $Y$ such that $F(Y,\tau_0) = \amal{F(H_N)}{F(H_{N'})}$. Let us distinguish the edges in $F(H_N)$ from those in $F(H_{N'})$ by adding a tick to those in $F(H_{N'})$, and doing the same for collections of edges.
Thus, $K_1 \in F(H_N)$, $K_1' \in F(H_{N'})$, and e.g.\ $C_1 = C_1'$, since these edges lie in the amalgamation base. By Fact~\ref{fact:identity-tree}, without loss of generality we may assume $\tau_0(\mathcal{T} \cup \mathcal{T}') = \{\id\}$. Then, let $\tau = \tau_0|_{H_{n,N}}$, $\tau' = \tau_0|_{H_{n,N'}}$ be the two natural restrictions of $\tau_0$.
By switching, we can assume that $\tau(\mathcal{C}_{\id}) = \id$.
Let $h_{s}$ stand for the element $\tau(\mathcal{C}_s)$, and
let $h_{j}$ stand for the element $\tau(\mathcal{C}_j)$ for all $j$.
For each $j \in \{1, \ldots, n\}$, by considering the balanced cycle $C_j^{-1}\ \delta_1\beta_1\alpha_1 \ A_j \ \alpha_2 \beta_2 \delta_2$ (where we omit an edge if it is in $\mathcal{T}$), we must have $\tau(A_j) = \tau(C_j) = h_j$.
Similarly, we must have $\tau(A_s) = h_s$.
Next, observe that since we have the balanced cycle
\[\alpha_1 \alpha_2 \beta_2 \gamma_{N+1} \gamma_{N} \cdots \gamma_1\ K_2^{-1}\ \beta_1\ K_1,\]
we must have $\tau(K_1) = \tau(K_2)$.
For each $1 \leq i \leq N$ and each $1 \leq j \leq n$, we consider the balanced cycle
\[\alpha_1\ A_i\ \alpha_2 \beta_2 \gamma_{N+1} \gamma_{N} \cdots \gamma_{i+1}\ B_{i,j}^{-1}\ \gamma_i \gamma_{i-1} \cdots \gamma_1 \ K_2^{-1}\ \beta_1\ K_1\ \alpha_1,\]
and conclude $\tau(B_{i,j}) = \tau(A_j) = h_j$.
Now, let $\eta:\{1, \ldots, N\} \rightarrow \{h_1, \ldots, h_n\} \cup \{\id\}$ be a string of length at most $N$ in the $h_i$. There is a path $P_\eta$ from $\gamma_1$ to $\gamma_{N+1}$ such that for each $k$, the edge $P_\eta$ takes between $\gamma_{k}$ and $\gamma_{k+1}$ is mapped to $\eta(k)$ by $\tau$.
Let $C_\eta$ be the cycle obtained from $P_\eta$ by appending $\gamma_{N+1} \beta_2 \alpha_2 \ A_s^{-1} \ \alpha_1 \ K_1^{-1}\ \beta_1 \ K_2 \ \gamma_1$.
Thus, $\tau(C_\eta) = \tau(P_\eta) h_s^{-1}$. Since $f_{\{a_i\}}(s) = N$, there exists $\eta$ a string of length $N$ such that $\sigma_N(C_\eta)=\id$, but not one of smaller length.
So there exists $\eta$ of length $N$ such that $\tau(C_\eta)=\id$ but not one of smaller length. In other words, $f_{\{h_i\}}(h_s) = N$. Symmetrically, if we let $\tau'(C_j) = h_j'$ for all $j$, and $\tau'(C_s) = h_s'$, then $f_{\{h_i'\}}(h_s') = N'$. 
But the edges in $\mathcal{C}$ are all in the amalgamation base, so $\tau$ and $\tau'$ agree on $\mathcal{C}$, so
\[N = f_{\{h_i\}}(h_s) = f_{\{h_i'\}}(h_s') = N'.\] Now we have shown the `only if' direction.

We prove the other direction. Fix some $N \geq 1$.
We must show that $\amal{F(H_N)}{F(H_N)}$ is a $\Gamma$\dash gain-graphic matroid.
As above, let us distinguish the two copies of $H_{n,N}$ by adding a tick to the names of the components of the second copy.
Thus, $K_1' \in H_{n,N'}$, $K_1' \not\in H_{n,N}$.
Let $\sigma_N'$ be the gaining $\sigma_N$, applied to $H_{n,N'}$.
Note that $\sigma_N$ and $\sigma_N'$ agree on the amalgamation base.

Let $Y$ be $H_{n,N}\oplus H_{n,N'}$.
Our goal is to use Lemma~\ref{lem:graph-amalgam} to construct a $\Gamma$\dash gaining of $Y$ such that the corresponding frame matroid is equal to $\amal{F(H_N)}{F(H_{N'})}$.
We observe that we cannot simply use $(H_{n,N} \oplus_{\delta_1,\delta_2} H_{n,N'},\sigma_N \cup \sigma_N')$, as for instance the cycle \[\delta_1 \beta_1 \ K_1 \ \alpha_1 \alpha_2 \beta_2 \delta_2 \beta_2' \alpha_2' \alpha_1'\  (K_1')^{-1} \ \beta_1' \delta_1\] (where, as before, we omit an edge if it is in $\mathcal{T} \cup \mathcal{T}'$) would be balanced, whereas that cycle is not balanced in the matroid amalgam.
Thus, we need our gaining to have different values on $K_1$ and $K_1'$.

Choose some $\mathfrak{M}' \in \Gamma$ such that $f_{\{a_i\}}(\mathfrak{M}') = 4N+2$. Let us define a gaining $\tau^{*}{}'$ of $H_{n,N'}$ which operates as follows:
\begin{align*}
     (A,u,v) &\mapsto \begin{cases}
     \sigma_N'(A,u,v) : A \in H_{n,N'}, A \not \in \{K_1',K_2'\}\\
     \mathfrak{M}' : (A,u,v) \in \{(K_1',\beta_1',\alpha_1'),(K_2', \beta_1',\gamma_1')\}\\
     (\mathfrak{M}')^{-1} : (A,u,v) \in \{(K_1',\alpha_1',\beta_1'),(K_2', \gamma_1',\beta_1')\}\\
     \end{cases}
\end{align*}
Then, let us extend this to $\tau'$ a $\Gamma$\dash gaining of $H_{n,N'}$ by first letting $\tau(Q_i') = \sigma_N(Q_i)$ for $i=N+6,N+7$, and then choosing the remaining values of $\tau(Q_i')$ such that no $Q_i'$ is a balanced loop, and then choosing values of $\tau(D_{i,j}')$ such that no cycle containing a $D_{i,j}'$ is balanced. We can do this since the group is infinite. Note $\tau'$ agrees with $\sigma_N$ on $\ell_N = \ell_{N'}$.

We claim $F(H_{n,N'},\tau') = F(H_{n,N'},\sigma_N')$.
Note both sides agree that any cycle containing a $D_{i,j}'$ or a $Q_i'$ is unbalanced. So the problem reduces to showing that
\[F(H^*_{n,N'},\tau^*{}') = F(H^*_{n,N'},\sigma{{}^*}_N').\]
Note that $\tau^*{}'$ disagrees with $\sigma{{}^*}'_N$ only on $K_1'$ and $K_2'$.
Further, since $K_{1}'$ and $K_{2}'$ are incident with a common vertex, and 
\[\tau^*{}'((K_1')^{-1} K_2') = \sigma{{}^*}_N'((K_1')^{-1} K_2') = \id,\]
the only cycles they could possibly disagree on are those containing exactly one of $K_1'$ and $K_2'$.
But for any such cycle, both $\tau^*{}'$ and $\sigma{{}^*}_N'$ agree that it is unbalanced. So as desired
\[F(H_{n,N'},\tau') = F(H_{n,N'},\sigma_N').\]
Now that we have this identity, we want to show $F(H_{n,N}\oplus_{\ell_N} H_{n,N'},\sigma_N \cup \tau')$ is the matroid $\amal{F(H_N)}{F(H_{N'})} = \amal{F(H_N)}{F(H_{n,N'},\tau')}$, by Lemma~\ref{lem:graph-amalgam}. We must prove that this data satisfies its conditions. Note (i), (ii) are immediate from the definition of $\tau'$, and (iii) is satisfied by $Q_{N+6}$ and $Q_{N+7}$, so there is only one condition of Lemma~\ref{lem:graph-amalgam} remaining: to show that for any paths $P \subseteq H_{n,N}$, $P' \subseteq H_{n,N'}$, both starting at $\delta_1$ and ending at $\delta_2$, and $\sigma_N(P) = \tau'(P')$, there is some edge $E$ from $\delta_1$ to $\delta_2$ such that $\sigma_N(E) = \sigma_N(P) = \tau'(P')$.
Fix such $P,P'$. Their union is a balanced cycle, so they must lie in $H^*_{n,N}$ and $H^*_{n,N'}$, respectively.
Then, in particular, $f_{\{a_i\}}(P) = f_{\{a_i\}}(P')$. By inspection $f_{\{a_i\}}(P) \leq 3N+1$.
If $P'$ passes through $K_1'$ or $K_2'$, $f_{\{a_i\}}(P') \geq 3N+2$. Therefore, $P'$ passes through neither $K_1'$ nor $K_2'$.
Hence, $\tau'(P') \in \{\id,a_1, \ldots, a_n, s\}$, and so there is some edge $E \in \mathcal{A}$ such that $\sigma_N(E) = \sigma_N(P) = \tau'(P')$.
The conditions of the Lemma are satisfied, so $\amal{F(H_N)}{F(H_{N'})} = F(H_{n,N}\oplus_{\ell_N} H_{n,N'},\sigma_N \cup \tau')$ is a $\Gamma$-gain graphic matroid. \end{proof}

In the next result, we show how Theorem~\ref{thm:uniformly-locally-finite} follows from the locally finite case (Theorem~\ref{thm:uniformly-locally-finite-2}) by using ultraproducts.

\begin{theorem}\label{thm:uniformly-locally-finite}
Let $\Gamma$ be an infinite group that is not uniformly locally finite.
The class of $\Gamma$\dash gain-graphic matroids is not \cmso\dash definable.
\end{theorem}

\begin{proof}
Assume $\Gamma$ is not uniformly locally finite.
Let $\mathcal{U}$ be a non-principal ultrafilter on $\bN$, and let $\Gamma^{\mathcal{U}}$ be the ultrapower of $\Gamma$.
Corollary~\ref{cor:ulf-iff-up-lf} says that $\Gamma^{\mathcal{U}}$ is not locally finite.
Therefore Theorem~\ref{thm:uniformly-locally-finite-2} implies that the class of $\Gamma^{\mathcal{U}}$\dash gain-graphic matroids is not \cmso\dash definable.
Corollary~\ref{cor:gain-matroids-are-same} now tells us that the class of $\Gamma$\dash gain-graphic matroids is not \cmso\dash definable.
\end{proof}

\section{The conviviality graph}
\label{sec:convivial}

Recall that a monomorphism is an injective homomorphism.

\begin{definition}[Conviviality graph]
\label{def:convivialty-graph}
We fix the finite group $F$.
Let $H$ be an infinite group which has a copy of $F$ as a subgroup.
Let $\mathcal{H}$ be the set of all pairs $(\Gamma,\psi)$, where
\begin{enumerate}[label = \textup{(\roman*)}]
\item $\Gamma$ is a finite group,
\item there exists a monomorphism from $\Gamma$ into $H$, and
\item $\psi$ is a monomorphism from $F$ into $\Gamma$.
\end{enumerate}
Define the equivalence relation $\sim$ on $\mathcal{H}$ such that $(\Gamma_{1},\psi_{1})\sim (\Gamma_{2},\psi_{2})$ if and only if there exists an isomorphism $\theta\colon \Gamma_{1}\to \Gamma_{2}$ such that $\psi_{2}=\theta\circ \psi_{1}$.
Say that $(\Gamma_{1},\psi_{1})$ and $(\Gamma_{2},\psi_{2})$ are representatives of two equivalence classes.
If there are monomorphisms $\theta_{1}\colon \Gamma_{1}\to H$ and $\theta_{2}\colon \Gamma_{2}\to H$ such that $\theta_{1}\circ \psi_{1} = \theta_{2}\circ \psi_{2}$, then we say that $(\Gamma_{1},\psi_{1})$ and $(\Gamma_{2},\psi_{2})$ are \emph{$F$\dash convivial} in $H$.
It is easy to see that the choice of representatives does not change whether the pair is $F$\dash convivial, so we can think of conviviality as being a relation on equivalence classes.
Note that every equivalence class is convivial with itself, since if $(\Gamma_{1},\psi_{1})\sim (\Gamma_{2},\psi_{2})$, then there is an isomorphism $\theta: \Gamma_1 \rightarrow \Gamma_2$ witnessing this, and then for any $\theta_2$, a monomorphism of $\Gamma_2$ into $H$, $\theta_2 \circ \theta$ is a monomorphism of $\Gamma_1$ into $H$, and $\theta_2 \circ \theta \circ \psi_1 = \theta_2 \circ \psi_2$, by the definition of $\theta$.
The \emph{elementary $F$\dash conviviality} graph of $\Gamma$ has the set of equivalence classes $\mathcal{H}/\!\!\sim$ as its vertex-set, where $(\Gamma_{1},\psi_{1})$ and $(\Gamma_{2},\psi_{2})$ are adjacent if and only if they are $F$\dash convivial.

We now define the equivalence relation $\approx$ on the vertices of the elementary conviviality graph so that two vertices are equivalent if they have exactly the same neighbours.
Note that this requires that the vertices are adjacent since every vertex is self-adjacent.
Now the \emph{$F$\dash conviviality graph} of $H$ has the equivalence classes of $\approx$ as its vertices.
Two equivalence classes are adjacent in the conviviality graph if and only if representative vertices from those classes are adjacent.
\end{definition}

\begin{definition}\label{def:conv-gadget}
Given finite groups $\Gamma_1 \leq \Gamma_2$, the graph $\Lambda^*_{\Gamma_1,\Gamma_2}$ is defined by Figure~\ref{fig:Lambda*G1G2}.\par
\begin{figure}[htb]
\centering
\begin {tikzpicture}[-latex ,auto,on grid , scale=0.8,
semithick ,
state/.style ={ circle ,top color =white , bottom color = processblue!20 ,
draw,processblue , text=blue , minimum width =1 cm}]
  \node [style1] (n11) at (18,0) {$\alpha_1$};
  \node [style1] (n12) at (22,0) {$\alpha_2$};
  \node [style1] (n21) at (13,-2)  {$\beta_1$};
  \node [style1] (n22) at (27,-2)  {$\beta_2$};
  \node [style1] (n31) at (18,-3.5)  {$\gamma_1$};
  \node [style1] (n32) at (20,-1.5)  {$\gamma_2$};
  \node [style1] (n33) at (22,-3.5)  {$\gamma_3$};
  \node [style1] (n41) at (18,-6) {$\delta_1$};
  \node [style1] (n42) at (22,-6)  {$\delta_2$};

\draw[gray,thick,dashed] ($(n41.north west)+(-0.7,1.0)$)  rectangle ($(n42.south east)+(0.7,-0.6)$);

\path (n11) edge [bend left =15] node[above] {$\mathcal{A}$} (n12);
\path (n11) edge [bend left =-15,very thick] node[above] {$\vdots$} (n12);
\path (n21) edge [bend left =25] node[above] {$K_1$} (n11);
\path (n21) edge [bend left =0,very thick] node[below] {} (n11);
\path (n12) edge [bend left =0,very thick] node[below] {} (n22);
\path (n21) edge [bend left =0] node[above] {$K_2$} (n31);
\path (n33) edge [bend left =0,very thick] node[above] {} (n22);
\path (n21) edge [bend left =0,very thick] node[above] {} (n41);
\path (n42) edge [bend left =0,very thick] node[above] {} (n22);

\path (n41) edge [bend left =-15,very thick] node[above] {$\vdots$} (n42);
\path (n41) edge [bend left =15] node[above] {$\mathcal{C}$} (n42);

\path (n32) edge [bend left =-25,very thick] node[above] {$\mathcal{B}_1$} (n31);
\path (n32) edge [bend left =25] node[above] {$\ddots$ \ \ } (n31);
\path (n33) edge [bend left =-25,very thick] node[above] {$\mathcal{B}_2$} (n32);
\path (n33) edge [bend left =25] node[above] {\ $\Ddots$} (n32);
\path (n31) edge [bend left =-25,very thick] node[above] {$\vdots$} (n33);
\path (n31) edge [bend left =0] node[above] {$\mathcal{B}_3$} (n33);

\end{tikzpicture}
\caption{The graph $\Lambda^{*}_{\Gamma_1,\Gamma_2}$}
\label{fig:Lambda*G1G2}
\end{figure}
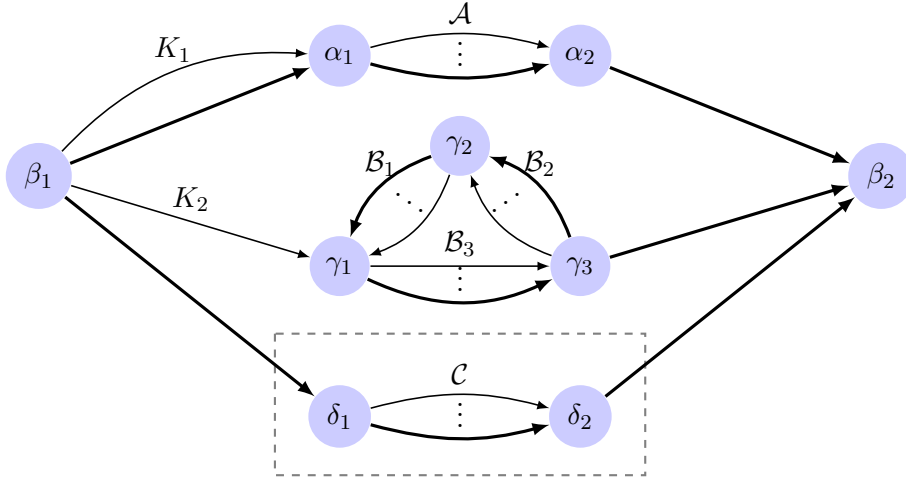
For ease of notation, we will refer to vertices by lower-case Greek letters, and edges by upper-case Roman letters. Let $\ell^*_{\Gamma_1}$ be the edge set of the subgraph enclosed by the dotted line (we will see its structure depends only on $\Gamma_1$). Let $\mathcal{T}$ be the collection of bolded edges. Note that each of $\mathcal{A},\mathcal{B}_1,\mathcal{B}_2,\mathcal{B}_3$ and $\mathcal{C}$ are collections of edges:
\begin{align*}
&\mathcal{A} = \{A_g : g \in \Gamma_1\}\\
&\mathcal{B}_i = \{B_{i,g} : g \in \Gamma_2\} : i=1,2,3\\
&\mathcal{C} = \{C_g : g \in \Gamma_1\}
\end{align*}
$\mathcal{T}$ has a single element in each of these: the elements $A_\id,B_{1,\id},B_{2,\id},B_{3,\id}$, and $C_\id$. Therefore, we have added a single bold line in each collection in the diagram. \par

Next, we define $\Lambda_{\Gamma_1,\Gamma_2}$ as the union of $\Lambda_{\Gamma_1,\Gamma_2}^*$ and the following new edges:
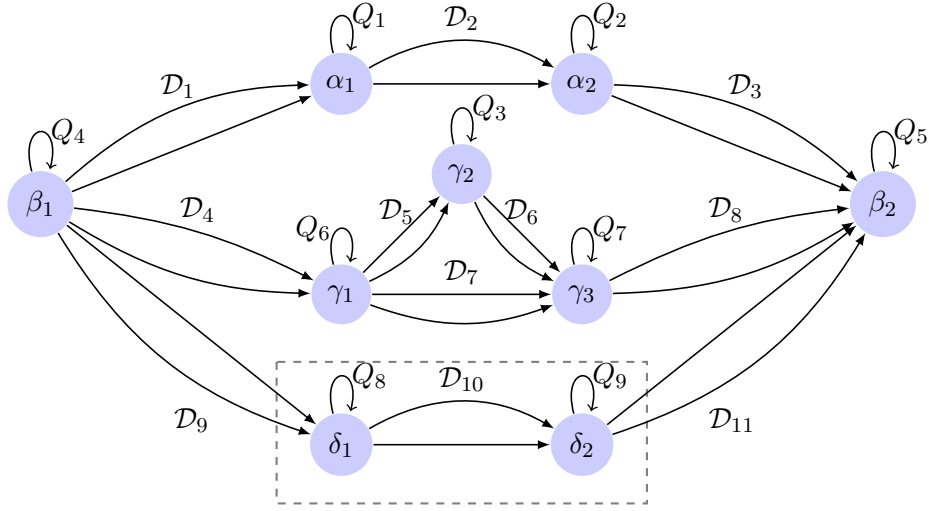
\begin{figure}[htb]
\centering
\begin {tikzpicture}[-latex ,auto,on grid , scale=0.8,
semithick ,
state/.style ={ circle ,top color =white , bottom color = processblue!20 ,
draw,processblue , text=blue , minimum width =1 cm}]
  \node [style1] (n11) at (18,0) {$\alpha_1$};
  \node [style1] (n12) at (22,0) {$\alpha_2$};
  \node [style1] (n21) at (13,-2)  {$\beta_1$};
  \node [style1] (n22) at (27,-2)  {$\beta_2$};
  \node [style1] (n31) at (18,-3.5)  {$\gamma_1$};
  \node [style1] (n32) at (20,-1.5)  {$\gamma_2$};
  \node [style1] (n33) at (22,-3.5)  {$\gamma_3$};
  \node [style1] (n41) at (18,-6) {$\delta_1$};
  \node [style1] (n42) at (22,-6)  {$\delta_2$};

\draw[gray,thick,dashed] ($(n41.north west)+(-0.7,1.0)$)  rectangle ($(n42.south east)+(0.7,-0.6)$);

\path (n21) edge [bend left =20] node[above] {$\mathcal{D}_1$} (n11);
\path (n21) edge [bend left =0] node[above] {} (n11);

\path (n11) edge [bend left =30] node[above] {$\mathcal{D}_2$} (n12);
\path (n11) edge [bend left =0] node[above] {} (n12);

\path (n12) edge [bend left =20] node[above] {$\mathcal{D}_3$} (n22);
\path (n12) edge [bend left =0] node[above] {} (n22);

\path (n31) edge [bend left =0] node[above] {$\mathcal{D}_5\text{ }$} (n32);
\path (n31) edge [bend left =-20] node[above] {} (n32);
\path (n32) edge [bend left =0] node[above] {$\mathcal{D}_6$} (n33);
\path (n32) edge [bend left =-20] node[above] {} (n33);
\path (n31) edge [bend left =0] node[above] {$\mathcal{D}_7$} (n33);
\path (n31) edge [bend left =-20] node[above] {} (n33);

\path (n21) edge [bend left =10] node[above] {$\mathcal{D}_4$} (n31);
\path (n21) edge [bend left =-15] node[above] {} (n31);

\path (n33) edge [bend left =10] node[above] {$\mathcal{D}_{8}$} (n22);
\path (n33) edge [bend left =-15] node[above] {} (n22);

\path (n21) edge [bend left =-20] node[below] {} (n41);
\path (n21) edge [bend left =0] node[below] {\begin{tabular}{c}
      \\
      \\
    $\mathcal{D}_9$
\end{tabular}} (n41);

\path (n42) edge [bend left =-20] node[below] {} (n22);
\path (n42) edge [bend left =0] node[below] {\begin{tabular}{c}
      \\
      \\
      $\mathcal{D}_{11}$
\end{tabular}} (n22);

\path (n41) edge [bend left =30] node[above] {$\mathcal{D}_{10}$} (n42);
\path (n41) edge [bend left =0] node[above] {} (n42);

\path (n11) edge [loop above] node[right] {$Q_1$} (n11);
\path (n12) edge [loop above] node[right] {$Q_2$} (n12);
\path (n21) edge [loop above] node[right] {$Q_4$} (n21);
\path (n22) edge [loop above] node[right] {$Q_5$} (n22);
\path (n31) edge [loop above] node[left] {$Q_6$} (n31);
\path (n32) edge [loop above] node[right] {$Q_3$} (n32);
\path (n33) edge [loop above] node[right] {$Q_7$} (n33);

\path (n41) edge [loop above] node[right] {$Q_8$} (n41);
\path (n42) edge [loop above] node[right] {$Q_{9}$} (n42);

\end{tikzpicture}
\caption{The graph $\Lambda_{\Gamma_1,\Gamma_2} \setminus \Lambda^{*}_{\Gamma_1,\Gamma_2}$}
\label{fig:LambdaG1G2}
\end{figure}

Each $\mathcal{D}_i : 1 \leq i \leq 11$ is a collection of two edges: $\mathcal{D}_i = \{D_{i,1}, D_{i,2}\}$. Each $Q_i$ is a single loop edge. Let $\ell_{\Gamma_1} = \ell^*_{\Gamma_1} \cup \{Q_8,Q_{9}\} \cup \mathcal{D}_{10}$. We will use these new edges at only one point in the argument.

Suppose we have an infinite group $\Gamma_3$ containing $\Gamma_2$, and $\mathfrak{M}$ some element in $\Gamma_3$. Then, $\Lambda^*_{\Gamma_1,\Gamma_2}$ has a $\Gamma_2$\dash gaining $\sigma^* = \sigma^*(\Gamma_1,\Gamma_2,\mathfrak{M})$, defined as follows:
\begin{align*}
    &\sigma^*(T) = \id \ \forall \ T \in \mathcal{T}\\
    &\sigma^*(A_g) = g \ \forall \ g \in \Gamma_1\\
    &\sigma^*(B_g) = g \ \forall \ g \in \Gamma_2\\
    &\sigma^*(C_g) = g \ \forall \ g \in \Gamma_1
\end{align*}
And $\sigma^*(K_1) = \sigma^*(K_2)=\mathfrak{M}$. Note all these edges are oriented as in the diagrams. Secondly, given elements $d_{i,j} : 1 \leq i \leq 11, 1 \leq j \leq 2$, and $q_i : 1 \leq i \leq 9$ of $\Gamma$, we can define $\sigma = \sigma(\Gamma_1,\Gamma_2,\mathfrak{M}, \{d_{i,j}\}, \{q_i\})$, a $\Gamma_3$\dash gaining of $H_{n,N}$ extending $\sigma^*$, by
\begin{align*}
    &\sigma(D_{i,j}) = d_{i,j} \ \forall \ 1 \leq i \leq 11, 1 \leq j \leq 2\\
    &\sigma(Q_i) = q_i \ \forall \ 1 \leq i \leq 9
\end{align*}
All these edges are also oriented as in the diagrams.
We will always assume
\begin{equation}
\text{\begin{tabular}{c}
    $\mathfrak{M} \not \in \Gamma_2$, and\\
     the $\{d_{i,j}\}$, and $\{q_i\}$ are chosen so that \\
     no cycle containing a $D_{i,j}$ or $Q_i$ is balanced
\end{tabular}} \tag{$\dagger$}
\end{equation}
Subject to this assumption, the balanced cycles of $\sigma$ depend only on the isomorphism type of the pair $(\Gamma_1,\Gamma_2)$.
\end{definition}

Note that by Remark~\ref{rmk:can-find-monsters}, we can always find such $\{d_{i,j}\}$ and $\{q_i\}$.

\begin{proposition}\label{prop:diagonal-conviviality-graph}
Let $\Gamma$ be an infinite group, and let $\Gamma_0$ be a finite subgroup.
Fix two representatives $(H_{1},\phi_{1})$ and $(H_{2},\phi_{2})$ of vertices of the elementary $\Gamma_0$\dash conviviality graph.
Fix, for each $k=1,2$, some data $(\mathfrak{M}^{(k)},\{d_{i,j}^{(k)}\},\{q_i^{(k)}\})$ satisfying condition $(\dagger)$ in \textup{Definition~\ref{def:conv-gadget}} with respect to the pair $(\Gamma_0,H_{k})$, and let $(\Lambda_{k},\sigma_k)$ be the gained graph
\[(\Lambda_{\phi_k(\Gamma_0),H_k},\sigma(\phi_k(\Gamma_0),H_k,\mathfrak{M}^{(k)},\{d_{i,j}^{(k)}\},\{q_i^{(k)}\})).\]
We assume $\ell_{\phi_1(\Gamma_0)} = \ell_{\phi_2(\Gamma_0)}$.
Then $(H_{1},\phi_{1})$ and $(H_{2},\phi_{2})$ are $\Gamma_0$\dash convivial in $\Gamma$ if and only if the matroid amalgam of $F(\Lambda_1,\sigma_1)$ and $F(\Lambda_2,\sigma_2)$ over $\ell_{\phi_1(\Gamma_0)}$ is $\Gamma$-gainable.
\end{proposition}
\begin{proof}
Fix $(H_1,\phi_1)$, $(H_2,\phi_2)$, $\mathfrak{M}^{(k)}$, $\{d_{j,k}^{(k)}\}$ and $\{q_j^{(k)}\}$, and $(\Lambda_1,\sigma_1)$, $(\Lambda_2,\sigma_2)$ as above.
For each $k=1,2$, let
\[(\Lambda_k^*,\sigma_k^*) = (\Lambda^*_{\phi_k(\Gamma_0),H_k},\sigma^*(\phi_k(F),H_k,\mathfrak{M}^{(k)})).\]
As in Theorem ~\ref{thm:uniformly-locally-finite-2}, let us distinguish the edges in $\Lambda_1$ and $\Lambda_2$ by adding a tick to those edges in $\Lambda_2$.
We have assumed that $\ell_{\sigma_1(\Gamma_0)}$ and $\ell_{\sigma_2(\Gamma_0)}$ are equal.
In addition, we assume $V(\Lambda_1) \cap V(\Lambda_2) = \{\delta_1,\delta_2\}$.
Therefore, for example, $Q_8 = Q_8'$ and $\delta_1=\delta_1'$.
Whenever we gain-graph amalgamate in this proof, it will be over the base $\{\delta_1,\delta_2\}$. Whenever we matroid amalgamate, it will be over the base $\ell_{\phi_1(\Gamma_0)}$, which is the set of edges adjacent only to $\delta_1$ and $\delta_2$.

The proof has two directions. First, suppose the matroid amalgam is $\Gamma$\dash gain-graphic. Let this be witnessed by a graph $X$ with $\Gamma$\dash gaining $\tau_0$, without loss of generality with no isolated vertices. 
As in Theorem ~\ref{thm:uniformly-locally-finite-2}, we wish to apply Lemma~\ref{lem:amalgam-is-biased}. We note that $(\Lambda_1,\sigma_1)$ and $(\Lambda_2,\sigma_2)$ have no balanced loop edges. The desired unbalanced loops are supplied by $Q_1, \ldots, Q_{10},Q_1', \ldots, Q_{10}'$. The desired unbalanced cycles are supplied by the $\mathcal{D}_i,\mathcal{D}_i'$. By definition, $X$ has no isolated vertices and
\[F(X,\tau_0) \cong \amal{F(\Lambda_1,\sigma_1)}{F(\Lambda_2,\sigma_2)}[\ell_{\phi_1(\Gamma_0)}].\]
Thus Lemma~\ref{lem:amalgam-is-biased} implies $X \cong \Lambda_1 \oplus \Lambda_2$.

Define $\tau$ to be the restriction $\tau_0|_{\Lambda_1}$, and let $\tau'$ be $\tau_0|_{\Lambda_2}$.
By Fact~\ref{fact:identity-tree}, we may assume that $\tau_0(\mathcal{T} \cup \mathcal{T}') = \{\id\}$.
For any fixed element $g \in \phi_1(\Gamma_0)$, by considering the balanced cycle $\delta_1\ C_g\ \delta_2 \beta_2 \alpha_2\ A_g\ \alpha_1 \beta_1 \delta_1$ (where we omit an edge if it is in $\mathcal{T}$), we deduce that $\tau(C_g) = \tau(A_g)$.
Similarly, the balanced cycle $\alpha_1\ A_\id\ \alpha_2 \beta_2 \gamma_3\ B_{3,\id}\ \gamma_1\ K_2\ \beta_1\ K_1$ implies that $\tau(K_1) = \tau(K_2)$, and also the balanced cycle $\alpha_1\ A_g\ \alpha_2 \beta_2 \gamma_3\ B_{3,g}\ \gamma_1\ K_2\ \beta_1\ K_1$ implies that $\tau(B_{3,g}) = \tau(A_g)$.
Finally, by considering the list of balanced cycles $\gamma_1\ B_{3,g}\ \gamma_3\ B_{2,h}\ \gamma_2\ B_{1,g+h}$, and noting that $\tau(B_{i,\id}) = \id$ for all $i$, we deduce that for all $g \in \phi_1(\Gamma_0)$, and all $i,j$, $\tau(B_{i,g}) = \tau(B_{j,g})$, and also for all $g,h \in H_1$, $\tau(B_{1,g})\tau(B_{1,h}) = \tau(B_{1,gh})$.
Similarly for all $g \in \phi_2(\Gamma_0)$, $\tau'(A'_g) = \tau'(B'_{3,g}) = \tau'(C'_g)$; for all $g \in H_2$, and all $i,j$, $\tau'(B_{i,g}') =\tau'(B_{j,g}')$; for all $g,h \in H_2$, $\tau'(B_{i,g})\tau'(B_{i,h}) = \tau'(B_{i,gh})$. For $i = 1,2$, let $\chi_i : H_i \rightarrow \Gamma$ be the function defined by
\[
             g \mapsto \begin{cases}
                 \tau(B_g) : i = 1\\
                 \tau'(B'_g) : i=2
             \end{cases}
\]
By the above results, for each $i$, $\chi_i$ is an embedding of $H_i$ into $\Gamma$, and since for all $g \in \Gamma_0$ we have $C_{\phi_1(g)} = C_{\phi_2(g)}'$, it follows that $\chi_1 \circ \phi_1 = \chi_2 \circ \phi_2$. Thus $(H_1,\phi_1)$ and $(H_2,\phi_2)$ are $\Gamma_0$\dash convivial in $\Gamma$.

Now we prove the other direction. Let $(H_1,\phi_1)$ and $(H_2,\phi_2)$ be $\Gamma_0$\dash convivial in $\Gamma$, witnessed by embeddings $\chi_i : H_i \rightarrow \Gamma$ such that $\chi_1 \circ \phi_1 = \chi_2 \circ \phi_2$.
For ease of notation, let us identify $H_i$ with its image under $\chi_i$ for each $i$, so now each $H_i$ is a specific subgroup of $\Gamma$. Note that since $\chi_1 \circ \phi_1 = \chi_2 \circ \phi_2$, this is consistent with our assumption $\ell_{\phi_1(\Gamma_0)} = \ell_{\phi_2(\Gamma_0)}$. We will define a $\Gamma$\dash gaining $\tau_0$ on $\Lambda_1 \oplus \Lambda_2$ such that $F(\Lambda_1 \oplus \Lambda_2,\tau_0) \cong \amal{F(\Lambda_1,\sigma_1)}{F(\Lambda_2,\sigma_2)}$.

As in Theorem~\ref{thm:uniformly-locally-finite-2}, we need to change the gaining of a few edges to make sure there are no `unintentional' balanced cycles.

Fix some $\mathfrak{M}'_1 \in \Gamma \setminus \langle H_1,H_2 \rangle$, $\mathfrak{M}'_2 \in \Gamma \setminus \langle H_1,H_2,\mathfrak{M}'_1 \rangle$. Let
\[\Delta_+ = \{(\gamma_1,\gamma_2),(\gamma_2,\gamma_3),(\gamma_3,\gamma_1)\}.\]
Then define a partial gaining $\tau^*$ which maps $(E,u,v)$ to the following elements:
\[
\begin{cases}
        \id : E \in \mathcal{T}\\
        g : E \in\{ B_{1,g},B_{2,g},B_{3,g}\}\ \text{ for some } g \in H_1, (u,v) \in \Delta_+\\
        \phi_1(g) : (E,u,v) \in\{ (A_{\phi_1(g)},\alpha_1,\alpha_2),(C_{\phi_1(g)},\delta_1,\delta_2)\} \text{ for some } g \in \Gamma_0\\
        \mathfrak{M}'_1 : (E,u,v) \in \{(K_1,\beta_1,\alpha_1),(K_2,\beta_1,\gamma_1)\}
    \end{cases}
\]
We extend it to a full gaining using the rule $\tau^*(E,u,v) = \tau^*(E,v,u)^{-1}$.
Let $\tau{{}^*}'$ be defined symmetrically, with $\tau{{}^*}'(K_1') = \tau{{}^*}'(K_2') = \mathfrak{M}'_2$. Note they agree on $\ell_{\phi_1(\Gamma_0)}$, the amalgamation base. Thus, we can define $\tau^*_0 = \tau^* \cup \tau{{}^*}'$. Next, we define $\tau_0$ by extending $\tau_0^*$ such that any cycle containing any $D_{i,j},D'_{i,j},Q_i$ or $Q_i'$ is unbalanced. We can do this since $\Gamma$ is infinite.
Now we must show $F(\Lambda_1 \oplus \Lambda_2,\tau_0)$ is as desired.
\begin{claim}
    $F(\Lambda_1, \tau_0|_{\Lambda_1}) \cong F(\Lambda_1,\sigma_1)$, and similarly for $\Lambda_2$
\end{claim}
\begin{proof}
    By symmetry, it suffices to prove the result for $\Lambda_1$. Note that the LHS and RHS both agree that any cycle containing any $D_{i,j}$ or $Q_i$ is unbalanced, so it suffices to check $F(\Lambda^*_1, \tau_0|_{\Lambda^*_1}) \cong F(\Lambda^*_1,\sigma^*_1)$.
    But note that $\tau_0|_{\Lambda^*_1} = \sigma^*(\phi_1(\Gamma_0),H_1,\mathfrak{M}'_1)$.
    Since $\mathfrak{M}'_1 \not \in H_1$, the balanced cycles induced by this gaining depend only on the isomorphism type of $(\phi_1(\Gamma_0),H_1)$ (as noted in Definition~\ref{def:conv-gadget}). Thus, since $(\phi_1(\Gamma_0),H_1) \cong (\phi_1(\Gamma_0),H_1)$, and $\sigma_1^* = \sigma^*(\phi_1(\Gamma_0),H_1,\mathfrak{M}^{(1)})$ we have the desired isomorphism.
\end{proof}
This implies
\begin{align*}
\amal{F(\Lambda_1,\sigma_1)}{F(\Lambda_2,\sigma_2)} \cong \amal{F(\Lambda_1,\tau_0|_{\Lambda_1})}{F(\Lambda_2,\tau_0|_{\Lambda_2})}    
\end{align*}
Now, we aim to apply Lemma~\ref{lem:graph-amalgam}, with $\Omega_i = (\Lambda_1,\tau_0|_{\Lambda_i})$ for $i=1,2$. The conclusion of that lemma would imply
\[\amal{F(\Lambda_1,\tau_0|_{\Lambda_1})}{F(\Lambda_2,\tau_0|_{\Lambda_2})} \cong F(\Lambda_1 \oplus \Lambda_2,\tau_0),\]
proving the proposition. So it suffices to check that for any two paths $P_1 \subseteq \Lambda_1$, $P_2 \subseteq \Lambda_2$, each going from $\delta_1$ to $\delta_2$, if $P_1 \cup P_2$ is balanced, then there is an edge $E$ going from $\delta_1$ to $\delta_2$ such that $\tau_0(P_1)=\tau_0(P_2)=\tau_0(E)$. Fix some such $P_1$, $P_2$.
Note that if either contains any $D_{i,j},D'_{i,j},Q_i$ or $Q_i'$, then $P_1 \cup P_2$ is automatically unbalanced, which would yield a contradiction. Since $\mathfrak{M}'_2 \not \in \langle H_1,H_2,\mathfrak{M}'_1 \rangle$, by the same reasoning we know $K_1',K_2' \not \in P_2$. Similarly $K_1,K_2 \not\in P_1$.
By inspecting Definition~\ref{def:conv-gadget}, we deduce that $\tau_0(P_1) \in \phi_1(\Gamma_0)$, since it can only consist of an element of $\mathcal{C}$, or an element of $\mathcal{A}$ and some elements of $\mathcal{T}$. Let $\tau_0(P_1) = \phi_1(g)$. Then $\tau_0(P_1) = \tau_0(P_2) = \tau_0(C_{\phi_1(g)})$, as desired.
Thus, we have satisfied the conditions of the lemma, so $\amal{F(\Lambda_1,\tau_0|_{\Lambda_1})}{F(\Lambda_2,\tau_0|_{\Lambda_2})} \cong F(\Lambda_1 \oplus \Lambda_1,\tau_0)$, which concludes the proof.
\end{proof}

\begin{theorem}\label{thm:conviviality-graph}
Let $\Gamma$ be a group.
If $\Gamma$ has a finite subgroup $\Gamma_0$ such that the $\Gamma_0$\dash conviviality graph of $\Gamma$ is infinite, then the class of $\Gamma$\dash gain-graphic matroids is not \cmso\dash definable.
\end{theorem}
\begin{proof}
Note that if $\Gamma$ is finite, then for any $\Gamma_0 \leq \Gamma$, the $\Gamma_0$\dash conviviality graph of $\Gamma$ must be finite, so the theorem is vacuously true. Therefore, suppose $\Gamma$ is infinite.
Fix the pairs $(H_1,\phi_1)$, $(H_2,\phi_2)$, representatives of vertices in the elementary $\Gamma_0$\dash conviviality graph of $\Gamma$. Also fix, for each $k=1,2$, some data \[(\mathfrak{M}^{(k)},\{d_{i,j}^{(k)}\},\{q_i^{(k)}\})\] satisfying condition $(\dagger)$ in Definition~\ref{def:conv-gadget} with respect to the pair $(\Gamma_0,H_{k})$, and let \[(\Lambda_{k},\sigma_k) = (\Lambda_{\phi_k(\Gamma_0),H_k},\sigma(\phi_k(\Gamma_0),H_k,\mathfrak{M}^{(k)},\{d_{i,j}^{(k)}\},\{q_i^{(k)}\})).\]
Then, by Proposition~\ref{prop:diagonal-conviviality-graph}, the amalgam $\amal{F(\Lambda_1,\sigma_1)}{F(\Lambda_2,\sigma_2)}$ is $\Gamma$\dash gain-graphic if and only if $(H_1,\phi_1)$, $(H_2,\phi_2)$ are $\Gamma_0$\dash convivial.
Note by Lemma~\ref{lem:Myhill-Nerode-application} if the class of $\Gamma$\dash gain-graphic matroids is \cmso\dash definable, there is a finite partition of $\{(H,\phi) \text{ in the elementary $\Gamma_0$\dash conviviality graph of $\Gamma$}\}$ such that whether $\amal{F(\Lambda_1,\sigma_1)}{F(\Lambda_2,\sigma_2)}$ is $\Gamma$\dash gain-graphic depends only on which classes $(H_1,\phi_1)$ and $(H_2,\phi_2)$ are in. But then, by the Proposition, whether $(H_1,\phi_1)$ and $(H_2,\phi_2)$ are $\Gamma_0$\dash convivial would depend only on which classes they are in, which would imply the $\Gamma_0$\dash conviviality graph is finite.
\end{proof}

\section{Acknowledgements}

We thank Sapir Ben-Shahar, Matthew Conder, Dugald Macpherson, Mike Newman, and Gabriel Verret for valuable conversations.


\begin{bibdiv}

\begin{biblist}

\bib{Aig97}{book}{
   author={Aigner, Martin},
   title={Combinatorial theory},
   series={Classics in Mathematics},
   note={Reprint of the 1979 original},
   publisher={Springer-Verlag, Berlin},
   date={1997},
   pages={viii+483}
}

\bib{vBDFGR15}{article}{
   author={van Bevern, Ren\'e},
   author={Downey, Rodney G.},
   author={Fellows, Michael R.},
   author={Gaspers, Serge},
   author={Rosamond, Frances A.},
   title={Myhill-Nerode methods for hypergraphs},
   journal={Algorithmica},
   volume={73},
   date={2015},
   number={4},
   pages={696--729}
}

\bib{Cou90}{article}{
   author={Courcelle, Bruno},
   title={The monadic second-order logic of graphs. I. Recognizable sets of
   finite graphs},
   journal={Inform. and Comput.},
   volume={85},
   date={1990},
   number={1},
   pages={12--75}
}

\bib{Euclid}{book}{
   author={Euclid},
   title={Elements},
   volume={IX},
   publisher={Megara},
   date={c.~300BCE}
   }

\bib{Fun15}{thesis}{
    title={On excluded minors and biased graph representations of frame matroids},
    author={Funk, Daryl},
    type={PhD thesis},
    date={2015},
    organization = {Simon Fraser University}
}

\bib{FMN21}{article}{
   author={Funk, Daryl},
   author={Mayhew, Dillon},
   author={Newman, Mike},
   title={Defining bicircular matroids in monadic logic},
   journal={Q. J. Math.},
   volume={73},
   date={2022},
   number={1},
   pages={65--92}
}

\bib{FMN22}{article}{
   author={Funk, Daryl},
   author={Mayhew, Dillon},
   author={Newman, Mike},
   title={Tree automata and pigeonhole classes of matroids: I},
   journal={Algorithmica},
   volume={84},
   date={2022},
   number={7},
   pages={1795--1834}
}

\bib{GGW13}{article}{
   author={Geelen, Jim},
   author={Gerards, Bert},
   author={Whittle, Geoff},
   title={Structure in minor-closed classes of matroids},
   conference={
      title={Surveys in combinatorics 2013},
   },
   book={
      series={London Math. Soc. Lecture Note Ser.},
      volume={409},
      publisher={Cambridge Univ. Press, Cambridge},
   },
   date={2013},
   pages={327--362}
}

\bib{GNW24}{article}{
   author={Geelen, Jim},
   author={Nelson, Peter},
   author={Walsh, Zach},
   title={Excluding a line from complex-representable matroids},
   journal={Mem. Amer. Math. Soc.},
   volume={303},
   date={2024},
   number={1523},
   pages={v+91}
}

\bib{Hli03a}{article}{
   author={Hlin\v en\'y, Petr},
   title={Branch-width, parse trees, and monadic second-order logic for
   matroids (extended abstract)},
   conference={
      title={STACS 2003},
   },
   book={
     series={Lecture Notes in Comput. Sci.},
     volume={2607},
      publisher={Springer, Berlin},
   },
   date={2003},
   pages={319--330}
}


\bib{H97}{book}{
    author={Hodges, Wilfrid},
    title={A shorter model theory},
    publisher={Cambridge University Press},
    date={1997}
}

\bib{HU79}{book}{,
    author={Hopcroft, John E.},
    author={Ullman, Jeffrey D.},
    title={Introduction to automata theory, languages, and computation},
    series={Addison-Wesley series in computer science},
    publisher={Addison-Wesley},
    date={1979},
    pages={x+418}
}

\bib{KK82}{article}{
   author={Kahn, J.},
   author={Kung, J. P. S.},
   title={Varieties of combinatorial geometries},
   journal={Trans. Amer. Math. Soc.},
   volume={271},
   date={1982},
   number={2},
   pages={485--499}
}

\bib{Kra12}{article}{
   author={Kr\'al, Daniel},
   title={Decomposition width of matroids},
   journal={Discrete Appl. Math.},
   volume={160},
   date={2012},
   number={6},
   pages={913--923}
}

\bib{KM14}{article}{
   author={Kotek, Tomer},
   author={Makowsky, Johann A.},
   title={Connection matrices and the definability of graph parameters},
   journal={Log. Methods Comput. Sci.},
   volume={10},
   date={2014},
   number={4},
   pages={4:1, 33}
}

\bib{MNW18}{article}{
   author={Mayhew, Dillon},
   author={Newman, Mike},
   author={Whittle, Geoff},
   title={Yes, the `missing axiom' of matroid theory is lost forever},
   journal={Trans. Amer. Math. Soc.},
   volume={370},
   date={2018},
   number={8},
   pages={5907--5929}
}

\bib{Myh57}{article}{
   author={Myhill, John},
   title={Finite automata and the representation of events},
   journal={WADD Technical Report},
   volume={57},
   date={1957},
   pages={112--137}
}

\bib{Ner58}{article}{
   author={Nerode, A.},
   title={Linear automaton transformations},
   journal={Proc. Amer. Math. Soc.},
   volume={9},
   date={1958},
   pages={541--544}
}

\bib{Oxl11}{book}{
   author={Oxley, James},
   title={Matroid theory},
   series={Oxford Graduate Texts in Mathematics},
   volume={21},
   edition={2},
   publisher={Oxford University Press, Oxford},
   date={2011},
   pages={xiv+684}
   }

\bib{Str11}{article}{
   author={Strozecki, Yann},
   title={Monadic second-order model-checking on decomposable matroids},
   journal={Discrete Appl. Math.},
   volume={159},
   date={2011},
   number={10},
   pages={1022--1039}
}

\bib{Zas89}{article}{
   author={Zaslavsky, Thomas},
   title={Biased graphs. I. Bias, balance, and gains},
   journal={J. Combin. Theory Ser. B},
   volume={47},
   date={1989},
   number={1},
   pages={32--52},
}

\bib{Zel90}{article}{
   author={Zel\cprime manov, E. I.},
   title={Solution of the restricted Burnside problem for groups of odd
   exponent},
   language={Russian},
   journal={Izv. Akad. Nauk SSSR Ser. Mat.},
   volume={54},
   date={1990},
   number={1},
   pages={42--59, 221},
   translation={
      journal={Math. USSR-Izv.},
      volume={36},
      date={1991},
      number={1},
      pages={41--60}
   }
}

\bib{Zel91}{article}{
   author={Zel\cprime manov, E. I.},
   title={Solution of the restricted Burnside problem for $2$-groups},
   language={Russian},
   journal={Mat. Sb.},
   volume={182},
   date={1991},
   number={4},
   pages={568--592},
   translation={
      journal={Math. USSR-Sb.},
      volume={72},
      date={1992},
      number={2},
      pages={543--565}
   }
}

\end{biblist}

\end{bibdiv}
\end{document}